\documentclass[reqno]{amsart}
\usepackage{amssymb}
\usepackage{amsfonts}
\usepackage{amsthm}
\usepackage{amsmath}
\usepackage{hyperref}

\newcommand{\Z}{{\mathbb{Z}}}
\newcommand{\C}{{\mathbb{C}}}
\newcommand{\R}{{\mathbb{R}}}

\let\Im=\undefined\DeclareMathOperator*{\Im}{Im}

\DeclareMathOperator*{\loc}{loc}

\DeclareMathOperator{\osc}{osc}
\newcommand{\rad}{\text{rad}}

\newtheorem{theorem}{Theorem}[section]
\newtheorem{prop}[theorem]{Proposition}
\newtheorem{lemma}[theorem]{Lemma}

\newtheorem{corollary}[theorem]{Corollary}
\newtheorem{conjecture}[theorem]{Conjecture}
\newtheorem{proposition}[theorem]{Proposition}
\theoremstyle{definition}
\newtheorem{definition}[theorem]{Definition}
\theoremstyle{remark}
\newtheorem*{remark}{Remark}
\newtheorem*{remarks}{Remarks}

\newcounter{smalllist}
\newenvironment{SL}{\begin{list}{{\rm(\roman{smalllist})\hss}}{%
\setlength{\topsep}{0mm}\setlength{\parsep}{0mm}\setlength{\itemsep}{0mm}%
\setlength{\labelwidth}{2.0em}\setlength{\itemindent}{2.5em}\setlength{\leftmargin}{0em}\usecounter{smalllist}%
}}{\end{list}}

\newenvironment{CI}{\begin{list}{{\ $\bullet$\ }}{%
\setlength{\topsep}{0mm}\setlength{\parsep}{0mm}\setlength{\itemsep}{0mm}%
\setlength{\labelwidth}{0mm}\setlength{\itemindent}{0mm}\setlength{\leftmargin}{0mm}%
\setlength{\labelsep}{0mm}
}}{\end{list}}

\newcommand{\eps}{{\varepsilon}}

\begin{document}

\title[The cubic NLS in 2D with radial data]
{The cubic nonlinear Schr\"odinger equation in two dimensions with radial data}
\author{Rowan Killip}
\address{University of California, Los Angeles}
\author{Terence Tao}
\address{University of California, Los Angeles}
\author{Monica Visan}
\address{Institute for Advanced Study}

\begin{abstract}
We establish global well-posedness and scattering for solutions to the mass-critical nonlinear Schr\"odinger equation
$iu_t + \Delta u = \pm |u|^2 u$ for large spherically symmetric $L^2_x(\R^2)$ initial data;
in the focusing case we require, of course, that the mass is strictly less than that of the ground state.
As a consequence, we deduce that in the focusing case, any spherically symmetric blowup solution must
concentrate at least the mass of the ground state at the blowup time.

We also establish some partial results towards the analogous claims in other dimensions and without the
assumption of spherical symmetry.
\end{abstract}

\maketitle

\tableofcontents

\section{Introduction}

We primarily consider the Cauchy problem for the cubic nonlinear Schr\"odinger equation (NLS)
\begin{equation}\label{cubic nls}
i u_t + \Delta u = \mu |u|^{2} u
\end{equation}
in two space dimensions with $L_x^2$ initial data.  Here $\mu=\pm1$, with $\mu=+1$ known as the \emph{defocusing} case
and $\mu=-1$ as the \emph{focusing} case.  (Other non-zero values of $\mu$ can be reduced to these two
cases by rescaling the values of $u$.)

The cubic nonlinearity is the most common nonlinearity in applications.  It arises
as a simplified model for studying Bose--Einstein condensates \cite{ErdosYau,gross,pitaevskii},
Kerr media in nonlinear optics \cite{Kelly,Talanov}, and even freak waves in the ocean \cite{DT,WaveMotion}.

From a mathematical point of view, the cubic NLS in two dimensions is remarkable for being mass-critical.

\subsection{The mass-critical nonlinear Schr\"odinger equation}

In arbitrary dimensions, $d\geq 1$, the mass-critical (or \emph{pseudoconformal}) nonlinear Schr\"odinger equation is
given by
\begin{equation}\label{nls}
i u_t + \Delta u = F(u) \quad \text{with } F(u) := \mu |u|^{4/d} u.
\end{equation}
The name is a testament to the fact that there is a scaling symmetry (see Definition~\ref{D:sym})
that leaves both the equation and the mass invariant.  Mass is a term used in physics to represent
the square of the $L^2_x$-norm:
$$
M(u(t)) := \int_{\R^d} |u(t,x)|^2\,dx.
$$
For \eqref{nls}, this is a conserved quantity; see Theorem~\ref{local}.

Our main result is to construct global strong $L^2_x(\R^d)$ solutions to \eqref{nls}
for spherically symmetric initial data in two space dimensions, $d=2$.
Many of our arguments continue to hold in greater generality, namely, for arbitrary dimensions $d\geq 1$ and without
the assumption of spherical symmetry. We will discuss this more general problem whenever it does not disrupt the flow
of our argument.

Let us first make the notion of a solution more precise:

\begin{definition}[Solution] A function $u: I \times \R^d \to \C$ on a non-empty time interval $I \subset \R$
(possibly infinite or semi-infinite) is a \emph{strong $L^2_x(\R^d)$ solution} (or \emph{solution} for short) to \eqref{nls}
if it lies in the class $C^0_t L^2_x(K \times \R^d) \cap L^{2(d+2)/d}_{t,x}(K \times \R^d)$ for all compact $K \subset I$,
and we have the Duhamel formula
\begin{align}\label{old duhamel}
u(t_1) = e^{i(t_1-t_0)\Delta} u(t_0) - i \int_{t_0}^{t_1} e^{i(t_1-t)\Delta} F(u(t))\ dt
\end{align}
for all $t_0, t_1 \in I$.  Note that by Lemma~\ref{L:strichartz}, the condition $u\in L^{2(d+2)/d}_{t,x}$
implies that the second term in the Duhamel formula above exists as a weak integral in $L_x^2$.
Here, $e^{it\Delta}$ is the propagator for the free Schr\"odinger equation defined via the Fourier transform
$$ \hat f(\xi) := (2\pi )^{-d/2} \int_{\R^d} e^{- ix \cdot \xi} f(x)\, dx$$
by
$$ \widehat{e^{it\Delta} f}(\xi) = e^{-it|\xi|^2} \hat f(\xi).$$
We refer to the interval $I$ as the \emph{lifespan} of $u$. We say that $u$ is a \emph{maximal-lifespan solution}
if the solution cannot be extended to any strictly larger interval. We say that $u$ is a \emph{global solution} if $I = \R$.
\end{definition}
\begin{definition}[Blowup]\label{D:blowup}
We say that a solution $u$ to \eqref{nls} \emph{blows up forward in time} if there exists a time $t_0 \in I$ such that
$$ \int_{t_0}^{\sup I} \int_{\R^d} |u(t,x)|^{2(d+2)/d}\, dx \,dt = \infty$$
and that $u$ \emph{blows up backward in time} if there exists a time $t_0 \in I$ such that
$$ \int_{\inf I}^{ t_0} \int_{\R^d} |u(t,x)|^{2(d+2)/d}\, dx\, dt = \infty.$$
\end{definition}

\begin{remark}
The condition that $u$ is in $L^{2(d+2)/d}_{t,x}$ locally in time is natural.  This is the space that appears in the
original Strichartz inequality, \cite{Strichartz}.  As a consequence all solutions to the linear problem lie in this
space.  Existence of solutions to \eqref{nls} in this space is guaranteed
by the local theory discussed below; it is also necessary in order to ensure uniqueness of solutions in this local theory.
Solutions to \eqref{nls} in this class have been intensively studied, see for example
\cite{begout, bourg.2d, ck, cwI, caz, keraani, mv, tao-lens, compact, tvz-higher, tsutsumi}.
\end{remark}

Next, we recall some basic facts from the local theory.

\begin{definition}[Convergence of solutions]  Let $u^{(n)}: I^{(n)} \times \R^d \to \C$ be a sequence of solutions
to \eqref{nls}, let $u: I \times \R^d \to \C$ be another solution, and let $K$ be a compact time interval.
We say that $u^{(n)}$ \emph{converges uniformly to} $u$ on $K$ if we have $K \subset I$ and $K \subset I^{(n)}$ for
all sufficiently large $n$, and furthermore, $u^{(n)}$ converges strongly to $u$ in
$C^0_t L^2_x(K \times \R^d) \cap L^{2(d+2)/d}_{t,x}(K \times \R^d)$ as $n \to \infty$.
We say that $u^{(n)}$ \emph{converges locally uniformly to} $u$ if $u^{(n)}$ converges uniformly to $u$
on every compact interval $K \subset I$.
\end{definition}

The local theory for \eqref{nls} was worked out by Cazenave and Weissler \cite{cwI}.  They constructed local-in-time
solutions for arbitrary initial data in $L^2_x(\R^d)$; however, due to the critical nature of the equation,
the resulting time of existence depends on the profile of the initial data and not merely on its $L_x^2$-norm.
Cazenave and Weissler also constructed global solutions for small initial data.  We summarize their results in the theorem below.

\begin{theorem}[Local well-posedness, \cite{cwI, caz}]\label{local}
Given $u_0 \in L^2_x(\R^d)$ and $t_0 \in \R$, there exists a unique maximal-lifespan solution $u$ to \eqref{nls} with $u(t_0)=u_0$.
We will write $I$ for the maximal lifespan.  This solution also has the following properties:
\begin{CI}
\item (Local existence) $I$ is an open neighbourhood of $t_0$.
\item (Mass conservation) The solution $u$ has a conserved \emph{mass}
\begin{equation}\label{massdef}
M(u) = M(u(t)) := \int_{\R^d} |u(t,x)|^2\ dx.
\end{equation}
\item (Blowup criterion) If $\sup(I)$ is finite, then $u$ blows up forward in time; if $\inf(I)$ is finite,
then $u$ blows up backward in time.
\item (Continuous dependence) If $u^{(n)}_0$ is a sequence converging to $u_0$ in $L^2_x(\R^d)$ and
$u^{(n)}: I^{(n)} \times \R^d \to \C$ are the associated maximal-lifespan solutions,
then $u^{(n)}$ converges locally uniformly to $u$.
\item (Scattering) If $\sup(I)=+\infty$ and $u$ does not blow up forward in time, then $u$ scatters forward in time, that is,
there exists a unique $u_+ \in L^2_x(\R^d)$ such that $$ \lim_{t \to +\infty} \| u(t)-e^{it\Delta} u_+ \|_{L^2_x(\R^d)} = 0.$$
Similarly, if $\inf(I)=-\infty$ and $u$ does not blow up backward in time, then $u$ scatters backward in time, that is,
there is a unique $u_- \in L^2_x(\R^d)$ so that
$$ \lim_{t \to -\infty} \| u(t)-e^{it\Delta} u_- \|_{L^2_x(\R^d)} = 0.$$
\item (Spherical symmetry) If $u_0$ is spherically symmetric, then $u$ remains spherically symmetric for all time.
\item (Small data global existence) If $M(u_0)$ is sufficiently small depending on $d$, then $u$ is a global solution
which does not blow up either forward or backward in time.  Indeed, in this case
$$\int_{\R}\int_{\R^d}|u(t,x)|^{2(d+2)/d}\,dx\,dt \lesssim M(u).$$
\end{CI}
\end{theorem}

A variant of the local well-posedness theorem above is the following

\begin{lemma}[Stability, \cite{tvz}]\label{stab} Fix $\mu$ and $d$.  For every $A > 0$ and $\eps > 0$ there exists $\delta > 0$
with the following property: if $u: I \times \R^d \to \C$ is an approximate solution to \eqref{nls} in the sense that
\begin{equation*}
 \| iu_t + \Delta u - F(u) \|_{L^{2(d+2)/(d+4)}_{t,x}(I \times \R^d)} \leq \delta
\end{equation*}
and also obeys
$$\|u\|_{L_{t,x}^{2(d+2)/d}(I\times\R^d)}\leq A,$$
and $t_0 \in I$ and $v_0 \in L^2_x(\R^d)$ are such that
$$
\|u(t_0)-v_0\|_{L_x^2(\R^d)} \leq \delta,
$$
then there exists a solution $v: I \times \R^d \to \C$ to \eqref{nls} with $v(t_0) = v_0$ such that
$$\|u-v\|_{L_{t,x}^{2(d+2)/d}(I\times\R^d)} \leq \eps.$$
In particular, by the Strichartz inequality,
$$
\|u-v\|_{L_t^\infty L_x^2(I\times \R^d)}\lesssim \delta + \eps(A+\eps)^{4/d}.
$$
\end{lemma}

\begin{remark}
This generalizes the continuous dependence statement of Theorem~\ref{local}.
It also implies the existence and uniqueness of maximal-lifespan solutions in Theorem~\ref{local}.
Analogous stability results for the energy-critical NLS (in $\dot H_x^1(\R^d)$ instead of $L^2_x(\R^d)$, of course)
have appeared in \cite{ckstt:gwp, merlekenig, RV, TaoVisan}.
\end{remark}

The class of solutions to \eqref{nls} enjoys a large number of important mass-preserving symmetries.
We first discuss the symmetries which fix the initial surface $t=0$.  We employ the notations from \cite{compact}.

\begin{definition}[Symmetry group]\label{D:sym}
For any phase $\theta \in \R/2\pi \Z$, position $x_0 \in \R^d$, frequency $\xi_0 \in \R^d$, and scaling parameter $\lambda > 0$,
we define the unitary transformation $g_{\theta,x_0,\xi_0,\lambda}: L^2_x(\R^d) \to L^2_x(\R^d)$ by the formula
$$
[g_{\theta, \xi_0, x_0, \lambda} f](x) :=  \frac{1}{\lambda^{d/2}} e^{i\theta} e^{i x \cdot \xi_0 }
        f\bigl( \frac{x-x_0}{\lambda} \bigr).
$$
We let $G$ be the collection of such transformations. If $u: I \times \R^d \to \C$, we
define $T_{g_{\theta,\xi_0,x_0,\lambda}} u: \lambda^2 I \times \R^d \to \C$ where
$\lambda^2 I := \{ \lambda^2 t: t \in I \}$ by the formula
$$
[T_{g_{\theta, \xi_0, x_0, \lambda}} u](t,x) :=  \frac{1}{\lambda^{d/2}} e^{i\theta} e^{i x \cdot \xi_0 } e^{-it|\xi_0|^2}
            u\left( \frac{t}{\lambda^2}, \frac{x-x_0 - 2\xi_0 t}{\lambda} \right),
$$
or equivalently,
$$
[T_{g_{\theta, \xi_0, x_0, \lambda}} u](t) = g_{\theta - t |\xi_0|^2, \xi_0, x_0 + 2 \xi_0 t, \lambda}
            \Bigl(u\bigl(\frac{t}{\lambda^2}\bigr)\Bigr).
$$
Note that if $u$ is a solution to \eqref{nls}, then $T_{g}u$ is a solution to \eqref{nls} with initial data $g u_0$.

We also let $G_\rad\subset G$ denote the collection of transformations in $G$ which preserve spherical symmetry, or more explicitly,
$$
G_\rad := \{ g_{\theta,0,0,\lambda}: \theta \in \R/2\pi \Z; \lambda > 0 \}.
$$
\end{definition}

\begin{remark} One easily verifies that $G$ is a group with $G_\rad$ as a subgroup and that the map $g \mapsto T_g$
is a homomorphism. The map $u \mapsto T_g u$ maps solutions to \eqref{nls} to solutions with the same mass and
$L^{2(d+2)/d}_{t,x}$ norm as $u$, that is, $M(T_g u) = M(u)$ and
$$\| T_g u \|_{L^{2(d+2)/d}_{t,x}(\lambda^2 I\times\R^d)} = \|u\|_{L^{2(d+2)/d}_{t,x}(I\times\R^d)}.$$
Furthermore, $u$ is a maximal-lifespan solution if and only if $T_g u$ is a maximal-lifespan solution.
\end{remark}

\begin{lemma}[Further symmetries of solutions]\label{sym}  Let $u$ be a solution to \eqref{nls} with lifespan $I$.
\begin{CI}
\item (Time reversal) The function
\begin{equation}\label{time-reverse}
\tilde u(t,x) := \overline{u(-t,x)}
\end{equation}
is a solution to \eqref{nls} with lifespan $-I := \{ -t: t \in I \}$ and mass $M(\tilde u) = M(u)$.
If $u$ is a maximal-lifespan solution, then so is $\tilde u$.
\item (Time translation) For any $t_0 \in \R$, the function
\begin{equation}\label{time-translate}
u_{t_0}(t,x) := u(t+t_0,x)
\end{equation}
is a solution to \eqref{nls} with lifespan $I-t_0 := \{ t-t_0: t \in I \}$ and mass $M(u_{t_0}) = M(u)$.
If $u$ is a maximal-lifespan solution, then so is $u_{t_0}$.
\item (Pseudoconformal transformation) If $0 \not \in I$, then the function
\begin{equation}\label{pc}
v(t,x) := |t|^{-d/2} e^{i|x|^2/4t} u\Bigl( -\frac{1}{t},\frac{x}{t} \Bigr)
\end{equation}
is a solution to \eqref{nls} with lifespan $-1/I := \{ -1/t: t \in I \}$ and mass $M(v) = M(u)$.
\end{CI}
Furthermore, $u$, $\tilde u$, $u_{t_0}$, and $v$ have the same $L^{2(d+2)/d}_{t,x}$ norm on their respective lifespans,
and if $u$ is spherically symmetric then so are $\tilde u$, $u_{t_0}$, and $v$.
\end{lemma}

\begin{proof} Direct computation.
\end{proof}

\subsection{The scattering conjecture}

From Theorem \ref{local} we see that all maximal-lifespan solutions of sufficiently small mass are automatically global
and do not blow up either forward or backward in time. But in the focusing case $\mu=-1$ it is well known that
this assertion can fail for solutions of large mass.  In particular, if we define the \emph{ground state} to be the
unique\footnote{The existence and uniqueness of $Q$ was established in \cite{blions} and \cite{kwong} respectively.}
positive radial Schwartz solution $Q: \R^d \to \R^+$ to the elliptic equation
\begin{equation}\label{ground}
\Delta Q + Q^{1+4/d} = Q,
\end{equation}
then $u(t,x) := e^{it} Q(x)$ is a solution to \eqref{nls}. This shows that a solution of mass $M(Q)$ can blow up both
forward and backward in time in the sense of Definition~\ref{D:blowup}; moreover, by applying
the pseudoconformal transformation \eqref{pc}, one obtains a Schwartz solution of mass $M(Q)$ that blows up in finite time.

It is however widely believed that this ground state example is the minimal mass obstruction to global
well-posedness and scattering in the focusing case, and that no such obstruction exists in the defocusing case.
More precisely, we have

\begin{conjecture}[Global existence and scattering]\label{conj}
Let $d \geq 1$ and $\mu=\pm 1$.  In the defocusing case $\mu=+1$, all maximal-lifespan solutions to \eqref{nls}
are global and do not blow up either forward or backward in time.  In the focusing case $\mu=-1$, all maximal-lifespan
solutions $u$ to \eqref{nls} with $M(u) < M(Q)$ are global and do not blow up either forward or backward in time.
\end{conjecture}

\begin{remarks}
1. While this conjecture is phrased for $L^2_x(\R^d)$ solutions, it is equivalent to a scattering claim for smooth solutions;
see \cite{begout, carles, keraani, tao-lens}.  In \cite{tao-lens} (and also in the earlier work \cite{blue}),
it is also shown that the global existence and the scattering claims are equivalent in the $L^2_x(\R^d)$ category.

2. Let us reiterate that blowup refers to infinite spacetime norm.  As noted in Theorem~\ref{local}, finiteness of
the $L^{2(d+2)/d}_{t,x}$ Strichartz norm implies scattering.  By Lemma~\ref{stab}, it also implies (quantitative)
continuous dependence upon initial data and stabillity under external forcing.
To this one may add stability of well-posedness under perturbations of the equation; see \cite{tvz}.

3. By Theorem~\ref{local}, the conjecture is known when the solution $u$ has sufficiently small mass;
the interesting questions are when the mass is very large in the defocusing case $\mu = +1$ or close to
the mass $M(Q)$ of the ground state in the focusing case $\mu=-1$.
\end{remarks}

Conjecture~\ref{conj} has been the focus of much intensive study and several partial results
for various choices of $d, \mu$, and sometimes with the additional assumption of spherical symmetry.
The most compelling evidence in favour of this conjecture stems from results obtained under the assumption
that $u_0$ has additional regularity.  For the defocusing equation, it is easy to prove global well-posedness
for initial data in $H^1_x$; this  follows from the usual subcritical argument combined with the conservation
of mass and energy; see, for example, \cite{caz}. Recall that the energy is given by
\begin{align}\label{energy}
E(u(t)):=\int_{\R^d}\frac 12 |\nabla u(t,x)|^2 +\mu \frac{d}{2(d+2)}|u(t,x)|^{\frac{2(d+2)}d} \, dx.
\end{align}
Note that for general $L_x^2$ initial data, the energy need not be finite.

The focusing equation with data in $H^1_x$ was treated by Weinstein \cite{weinstein}.  A key ingredient
was his proof of the sharp Gagliardo--Nirenberg inequality:

\begin{theorem}[Sharp Gagliardo--Nirenberg, \cite{weinstein}]\label{sGN}
\begin{equation}
\int_{\R^d} \bigl|f(x)\bigr|^{\frac{2(d+2)}{d}} \, dx   \leq
 \frac{d+2}{d} \biggl( \frac{\|f\|_{L^2}^2}{\|Q\|_{L^2}^2} \biggr)^{\frac{2}{d}} \int_{\R^d} \bigl|\nabla f(x)\bigr|^{2} \, dx.
\end{equation}
\end{theorem}

As noticed by Weinstein, this inequality implies that the energy \eqref{energy} is positive once $M(u_0)<M(Q)$;
indeed, it gives an upper bound on the $\dot H_x^1$-norm of the solution at all times of existence.
Combining this with a contraction mapping argument and the conservation of mass and energy,
Weinstein proved global well-posedness for the focusing equation with initial data in $H^1_x$ and mass smaller than
that of the ground state.

Note that the iterative procedure used to obtain a global solution both for the defocusing and the focusing equations
with initial data in $H^1_x$ does not yield finite spacetime norms; in particular, scattering does not follow even
for more regular initial data.

There has been much work \cite{bourg.2d, CGTz, CKSTT2, resonant, CRSW, PSST-hD, PSST-1D, gf, tzirakis, VZ-focusing} devoted
to lowering the regularity of the initial data from $H_x^1$ toward $L^2_x(\R^d)$ and thus, toward establishing the conjecture.
For instance, in the defocusing case $\mu=+1$, $d=2$, the best known result is in \cite{CGTz}, where global well-posedness
is established for $u_0 \in H^s_x(\R^2)$ whenever $s > 2/5$.  In the focusing case $\mu=-1$, $d=2$, with $M(u_0) < M(Q)$,
the best result is in \cite{resonant}, which achieved the same claim for $s > 1/2$.  In \cite{blue} it was shown that such
global well-posedness results in the $H^s_x(\R^d)$ class lead (via the pseudoconformal transform \eqref{pc}) to global existence
and scattering results in the weighted space $L^2_x(\R^d; \langle x \rangle^{2s}\ dx)$; thus, for instance, when $\mu=+1, d=2$,
one has global existence and scattering whenever $u_0 \in L^2_x(\R^d; \langle x \rangle^{2s}\ dx)$ for some $s > 2/5$.
(Scattering results of this type were first obtained in \cite{tsutsumi}.)

For $d=1$ or $d=2$, Conjecture \ref{conj} would imply (by standard persistence of regularity theory) that for any $s \ge 0$
and for $u$ as in the conjecture, the $H^s_x(\R^d)$ norm of the solution $u(t)$ at an arbitrary time $t$
would be bounded by the $H^s_x(\R^d)$ norm of $u(t_0)$ for any fixed $t_0$, times a quantity depending only on
$M(u)$ and the dimension $d$, that is,
$$
\|u(t)\|_{H^s_x}\leq C(M(u),d)\|u(t_0)\|_{H_x^s}.
$$
In particular, these norms would be bounded uniformly in $t$.  In this direction, some polynomial upper bounds on the
growth in time of $H^s_x(\R^d)$ norms in the case $d=2$ were established in \cite{borg:sob, cdks, st, st2}.

In a slightly different direction, it was shown by Nakanishi \cite{nak} that one has global existence and scattering
for $H^1_x$ initial data in the defocusing case $\mu=+1, d=2$, whenever the mass-critical nonlinearity $|u|^2 u$ is
increased in strength to the mass-supercritical (but energy-subcritical) nonlinearity $|u|^{2+\eps} u$ for some $\eps > 0$.

In the case of spherically symmetric solutions, the conjecture was recently settled in the high-dimensional defocusing case
$\mu=+1$, $d \geq 3$ in \cite{tvz-higher}.  The focusing case and the non-spherically-symmetric case remain open in all
higher dimensions. Both \cite{tvz-higher} and the current paper build on techniques developed in order to treat the
energy-critical NLS; see \cite{borg:scatter, ckstt:gwp, merlekenig, RV, tao:gwp radial, thesis:art, Monica:thesis}.
We will better explain our debt to  this work when we outline our argument.  For the energy-critical problem, the analogue of
Conjecture~\ref{conj} is mostly settled, with the only currently outstanding problem being the focusing case $\mu=-1$ with
non-spherically-symmetric data.

\subsection{Main result}

Our main result settles the scattering conjecture in the spherically symmetric case in two dimensions:

\begin{theorem}\label{main}
Let $d=2$.  Then Conjecture~\ref{conj} is true (for either choice of sign $\mu$) whenever $u$
is spherically symmetric.
\end{theorem}

In particular, in the defocusing case one now has global well-posedness and scattering in the class $L^2_x(\R^2)_\rad$
of spherically symmetric $L^2_x(\R^2)$ functions for arbitrarily large mass, while in the focusing case one has a similar
claim under the additional assumption $M(u) < M(Q)$.

Neither Theorem~\ref{main} nor Conjecture~\ref{conj} address the focusing problem for masses greater than or equal to
that of the ground state.  In this case, blowup solutions exist and attention has been focused on describing their properties.
Finite-time blowup solutions with finite energy and mass equal to that of the ground state have been completely
characterized by Merle \cite{merle2}; they are precisely the ground state solution up to symmetries of the equation.

Several works have shown that finite-time blowup solutions must concentrate a positive
amount of mass around the blowup time $T^*$.  For finite energy data, Merle and Tsutsumi \cite{mt} (for radial data) and
Nawa \cite{nawa} and Weinstein \cite{weinstein-2} (for general data) proved the following: there exists $x(t)\in \R^d$ so that
$$
\liminf_{t\nearrow T^*}\int_{|x-x(t)|\leq R} \bigl|u(t,x)\bigr|^2 \,dx \geq M(Q)
$$
for any $R>0$.  For merely $L^2$ initial data, Bourgain \cite{bourg.2d} proved that some small amount of mass must
concentrate in parabolic windows (at least along a subsequence):
$$
\limsup_{t\nearrow T^*} \sup_{x_0\in\R^2} \int_{|x-x_0|\leq (T^*-t)^{1/2}} \bigl|u(t,x)\bigr|^2 \,dx \geq c,
$$
where $c$ is a small constant depending on the mass of $u$. This result was extended to other dimensions in \cite{begout,keraani}.
Note that by a construction of Perelman \cite{perelman}, the parabolic window cannot be made smaller by more than
the square root of a double-logarithmic factor.  Our contribution to this line of investigation is Corollary~\ref{cor:conc} below.

In a series of papers \cite{mr2, mr3, mr}, Merle and Raphael gave a more or less complete description of blowup
behaviour at masses slightly above that of the ground state (and finite energy).  In particular, they show that
the blowup rate observed by Perelman is generic.

Keraani \cite[Theorem~1.12]{keraani} also describes blowup behaviour for masses close to the critical mass.
(See also \cite[Theorem~3]{mv} for a closely related result.) In light of Theorem~\ref{main}
(and using the pseudoconformal transformation), his result implies

\begin{corollary}[Small blowup solutions concentrate mass $M(Q)$]
Let $u$ be a spherically symmetric solution to \eqref{cubic nls} with $M(u)< 2M(Q)$ that blows up forward in time.
If the blowup time $T^*$ is finite, then
$$
\liminf_{t\to T^*} \int_{|x|\leq R(t)} |u(t,x)|^2\, dx \geq M(Q)
$$
for any function $R(t)$ obeying $(T^*-t)^{-1/2} R(t)\to\infty$ as $t\to T^*$.
If $T^*=\infty$, then
$$
\liminf_{t\to \infty} \int_{|x|\leq R(t)} |u(t,x)|^2\, dx \geq M(Q)
$$
for any function $R(t)$ obeying $t^{-1/2} R(t)\to\infty$ as $t\to \infty$.
A similar statement holds in the negative time direction.
\end{corollary}

As discussed in \cite{keraani,mv}, the main obstacle to treating general masses is the spectre of quadratic oscillation. We will
(partially) circumvent this problem using some ideas from \cite{bourg.2d}.  The price we must pay is to replace $\liminf$ by
$\limsup$, just as in \cite{bourg.2d}.

\begin{corollary}[Blowup solutions concentrate the mass of the ground state]\label{cor:conc}\leavevmode\\
Let $u$ be a spherically symmetric solution to \eqref{cubic nls} that blows up at time $0<T^*\leq \infty$.
If $T^*<\infty$, then there exists a sequence $t_n \nearrow T^*$ so that for any sequence
$R_n\in (0,\infty)$ obeying $(T^*-t_n)^{-1/2} R_n\to \infty$,
\begin{align}\label{conc finite}
\limsup_{n\to \infty}\int_{|x|\leq R_n}|u(t_n,x)|^2\, dx\geq M(Q).
\end{align}
If $T^*=\infty$, then there exists a sequence $t_n \to\infty$ such that for any sequence $R_n\in (0,\infty)$ with
$t_n^{-1/2} R_n\to \infty$,
\begin{align}\label{conc infinite}
\limsup_{n\to \infty}\int_{|x|\leq R_n}|u(t_n,x)|^2\, dx\geq M(Q).
\end{align}
The analogous statement holds in the negative time direction.
\end{corollary}

\begin{remark}
The argument used to deduce Corollary~\ref{cor:conc} works in any dimension and without the assumption of spherical symmetry;
it just relies on an affirmative answer to Conjecture~\ref{conj}.  However, as this remains open in such generality, we present
the argument in context of Theorem~\ref{main}.
\end{remark}
The arguments in this paper rely heavily on the spherical symmetry in many places, for example, in order to localize the solution
at the spatial origin $x=0$ and frequency origin $\xi=0$, and also to provide some strong spatial decay as $|x| \to \infty$.
We do not see how to remove this assumption, but it seems that if one wishes to do so, one should first look
at the higher-dimensional defocusing case $d \geq 3$, $\mu=+1$, in which the numerology is more favourable
(in particular, the dispersive inequality becomes more powerful) and one has additional tools such as
Morawetz inequalities available.

\subsection{Reduction to almost periodic solutions}

Bourgain's seminal work, \cite{borg:scatter}, on the energy-critical NLS first realized the important role played by solutions
that are simultaneously localized in both frequency and space.  The sufficiency of treating such solutions stems from his
`induction on energy' argument.  These ideas where pursued further in the work that followed,
for example, \cite{ckstt:gwp, RV, tao:gwp radial, thesis:art, Monica:thesis}.

A new and much more efficient approach to the energy-critical NLS (albeit non-quantitative) was introduced by Kenig and Merle,
\cite{merlekenig}.  They replace the induction on energy approach by a direct consideration of minimal-energy blowup solutions.
The existence of such solutions (in the mass-critical setting) is a profound observation of Keraani \cite{keraani}.
At a technical level, the new approach uses a concentration compactness result from \cite{keraani-2}.
Important related contributions include \cite{begout, bourg.2d, mv}.

In the Schr\"odinger context, the role of concentration compactness is to provide
a linear profile decomposition.  The main technical ingredients from the quantitative argument
(for example, refined Strichartz estimates) appear in the proof of such a decomposition theorem.
The two dimensional mass-critical version of the linear profile decomposition theorem reads as follows:

\begin{theorem}[Linear profiles, \cite{mv}]\label{lp}
Let $u_n$, $n=1,2,\ldots$ be a bounded sequence in $L^2_x(\R^2)$.  Then (after passing to a subsequence if necessary)
there exists a sequence of functions $\phi^j\in L^2_x(\R^2)$, group elements $g^{j}_n \in G$, and $t_n^j\in \R$
such that we have the decomposition
\begin{equation}\label{undecomp}
u_n = \sum_{j=1}^J g^j_n e^{it_n^j\Delta}\phi^j + w^J_n
\end{equation}
for all $J=1,2,\ldots$; here, $w^J_n \in L^2_x(\R^2)$ is such that its linear evolution has asymptotically vanishing scattering size:
\begin{equation}\label{sln}
\lim_{J \to \infty} \limsup_{n \to \infty} \|e^{it\Delta} w^J_n \|_{L_{t,x}^4} = 0.
\end{equation}
Moreover, for any $j \neq j'$,
\begin{align}\label{crazy}
\frac{\lambda_n^j}{\lambda_n^{j'}} + \frac{\lambda_n^{j'}}{\lambda_n^{j}}
        + \lambda_n^j \lambda_n^{j'}|\xi_n^j-\xi_n^{j'}|^2
        + \frac{|x_n^j-x_n^{j'}|^2}{\lambda_n^j \lambda_n^{j'}}
        + \frac{\bigl|t_n^j(\lambda_n^j)^2- t_n^{j'}(\lambda_n^{j'})^2\bigr|}{\lambda_n^j\lambda_n^{j'}}\to\infty.
\end{align}
Furthermore, for any $J \geq 1$ we have the mass decoupling property
\begin{equation}\label{un}
\lim_{n \to \infty} \Bigl[M( u_n ) - \sum_{j=1}^J M( \phi^j ) - M( w^J_n )\Bigr] = 0.
\end{equation}
Lastly, if $u_n$ are assumed to be spherically symmetric and $d \geq 2$, one can choose $\phi_j$ and $w_n^J$
to be spherically symmetric and $g_n^j\in G_\rad$.
\end{theorem}

\begin{remark}
Note that spherical symmetry is only a significant additional hypothesis when $d \geq 2$.
When $d=1$, any non-symmetric solution $u$ can be used to generate a spherically symmetric solution
of approximately twice the mass by considering the near-solution $u(t,x-x_0) + u(t,x+x_0)$
for some very large $x_0$ and using perturbation theory such as Lemma~\ref{stab} to convert this
to an actual solution; we leave the details to the reader.  Because of this, we do not expect the
spherically symmetric case to be significantly different from the non-symmetric case in one dimension.
\end{remark}

The concentration compactness approach reduces matters to the study of almost periodic solutions (modulo symmetries).
This reduction does not depend upon the dimension, nor does it require the assumption of spherical symmetry.
Indeed, in his  recent lectures \cite{Evian}, Kenig stressed that these ideas should be applicable to virtually
any dispersive equation with a good  local theory; how (and indeed whether) this helps in settling
global well-posedness questions is very much equation dependent.

\begin{definition}[Almost periodicity modulo symmetries]\label{apdef}
Let $d \geq 1$ and $\mu = \pm 1$.  A solution $u$ with lifespan $I$ is said to be \emph{almost periodic modulo $G$}
if there exist (possibly discontinuous) functions $N: I \to \R^+$, $\xi: I \to \R^d$, $x: I \to \R^d$ and a function
$C: \R^+ \to \R^+$ such that
$$ \int_{|x-x(t)| \geq C(\eta)/N(t)} |u(t,x)|^2\ dx \leq \eta$$
and
$$ \int_{|\xi-\xi(t)| \geq C(\eta) N(t)} |\hat u(t,\xi)|^2\ d\xi \leq \eta$$
for all $t \in I$ and $\eta > 0$.  We refer to the function $N$ as the \emph{frequency scale function} for the solution $u$,
$\xi$ as the \emph{frequency center function}, $x$ as the \emph{spatial center function}, and $C$ as the
\emph{compactness modulus function}.   Furthermore, if we can select $x(t) = \xi(t) = 0$, then we say that $u$
is \emph{almost periodic modulo $G_{\rad}$}.
\end{definition}

\begin{remarks}
1. The parameter $N(t)$ measures the frequency scale of the solution at time $t$, and $1/N(t)$ measures the spatial scale;
see \cite{compact, tvz-higher} for further discussion.  Note that we have the freedom to modify $N(t)$
by any bounded function of $t$, provided that we also modify the compactness modulus function $C$ accordingly.
In particular, one could restrict $N(t)$ to be a power of $2$ if one wished, although we will not do so here.
Alternatively, the fact that the solution trajectory $t \mapsto u(t)$ is continuous in $L^2_x(\R^d)$
can be used to show that the functions $N$, $\xi$, $x$ may be chosen to depend continuously on $t$.

2. One can view $\xi(t)$ and $x(t)$ as roughly measuring the (normalised) momentum and center-of-mass, respectively,
at time $t$, although as $u$ is only assumed to lie in $L^2_x(\R^d)$, these latter quantities are not quite rigorously defined.

3. By the Ascoli--Arzela Theorem, a family of functions is precompact in $L_x^2(\R^d)$ if and only if it is norm-bounded
and there exists a compactness modulus function $C$ so that
$$
\int_{|x| \geq C(\eta)} |f(x)|^2\ dx + \int_{|\xi| \geq C(\eta)} |\hat f(\xi)|^2\ d\xi \leq \eta
$$
for all functions $f$ in the family.  Thus, an equivalent formulation of Definition~\ref{apdef} is as follows:
$u$ is almost periodic modulo $G$ (respectively modulo $G_\rad$) if and only if there exists a compact subset $K$ of $L^2_x(\R^d)$
such that the orbit $\{ u(t): t \in I \}$ is contained inside $GK := \{ g f: g \in G, f \in K \}$ (respectively $G_\rad K$).
This may help explain the terminology `almost periodic'.
\end{remarks}

In \cite{compact} the following result was established (see also \cite{begout, keraani}),
showing that any failure of Conjecture \ref{conj} must be `caused' by a very special type of solution.

\begin{theorem}[Reduction to almost periodic solutions]\label{main-compact}
Fix $\mu$ and $d$ and suppose that Conjecture \ref{conj} failed for this choice of $\mu$ and $d$.
Then there exists a maximal-lifespan solution $u$ which is almost periodic modulo $G$ and which blows up
both forward and backward in time, and in the focusing case we also have $M(u) < M(Q)$.
If furthermore $d \geq 2$ and Conjecture \ref{conj} failed for spherically symmetric data,
then we can also ensure that $u$ is spherically symmetric and almost periodic modulo $G_\rad$.
\end{theorem}

\begin{proof}
See \cite[Theorem 1.13]{compact} and \cite[Theorem 7.2]{compact}.
\end{proof}

\begin{remark}
Theorem \ref{main-compact} was the first step in the resolution of Conjecture~\ref{conj}
in the case $d \geq 3$, $\mu=+1$ with spherical symmetry in \cite{tvz-higher}, just as the
analogous statement for the energy-critical NLS was the first step in \cite{merlekenig}.
\end{remark}

\subsection{Outline of the proof}
Under the assumption that Conjecture~\ref{conj} fails,
Theorem~\ref{main-compact} constructs an almost periodic solution $u$ with some frequency scale function $N(t)$,
but provides little information about the behaviour of $N(t)$ over the lifespan $I$ of the solution.
In this paper we refine Theorem~\ref{main-compact} by showing that the failure of Conjecture~\ref{conj} implies the
existence of one of three types of almost periodic solutions $u$ for which $N(t)$ and $I$ have very particular properties.
The argument is independent of the dimension, the sign of $\mu$, and makes no assumption of spherical symmetry.
Indeed,  the argument is predominantly combinatorial and, as with the reduction to almost periodic solutions,
should be applicable whenever there is a satisfactory local theory.

\begin{theorem}[Three special scenarios for blowup]\label{comp}
Fix $\mu$ and $d$ and suppose that Conjecture \ref{conj} fails for this choice of $\mu$ and $d$.
Then there exists a maximal-lifespan solution $u$ which is almost periodic modulo $G$, blows up both forward and backward in time,
and in the focusing case also obeys $M(u) < M(Q)$.
If furthermore $d \geq 2$ and Conjecture \ref{conj} fails for spherically symmetric data, then $u$ may be chosen
to be spherically symmetric and almost periodic modulo $G_\rad$.

With or without spherical symmetry, we can also ensure that the lifespan $I$ and the frequency scale function
$N:I\to\R^+$ match one of the following three scenarios:
\begin{itemize}
\item[I.] (Soliton-like solution) We have $I = \R$ and
\begin{equation}\label{nsim}
 N(t) = 1
\end{equation}
for all $t \in \R$ (thus the solution stays in a bounded space/frequency range for all time).
\item[II.] (Double high-to-low frequency cascade) We have $I = \R$,
\begin{equation}\label{cascade}
\liminf_{t \to -\infty} N(t) = \liminf_{t \to +\infty} N(t) = 0,
\end{equation}
and
\begin{equation}\label{freqbound}
\sup_{t \in \R} N(t) < \infty
\end{equation}
for all $t \in I$.
\item[III.] (Self-similar solution) We have $I = (0,+\infty)$ and
\begin{equation}\label{ssb}
 N(t) = t^{-1/2}
\end{equation}
for all $t \in I$.
\end{itemize}
\end{theorem}

\begin{remark}
Elementary scaling arguments show that one may assume that $N(t)$ is either bounded
from above or from below at least on half of its maximal lifespan.  This observation was used in \cite[Theorem 3.3]{tvz-higher}
and also in \cite{merlekenig, merle}.  However, the proof of Theorem~\ref{main} seems to require the
finer control given by Theorem \ref{comp} on the nature of the blowup as one approaches either endpoint of the interval $I$.
\end{remark}

The proof of Theorem~\ref{comp} can be found in Section~\ref{comp-sec}; it uses the basic
properties of almost periodic solutions developed in Section~\ref{apsec}.

Thus, by the end of Section~\ref{comp-sec} we have isolated our three main enemies (i.e. Scenarios I, II, and III); the remainder
of the paper is devoted to defeating them.  Our arguments to achieve this are very much specific to the
equation and depend heavily on the radial assumption.  In all three scenarios, the key step is to prove
that $u$ has additional regularity, indeed, more than one derivative in $L_x^2$.
In this regard, our approach differs somewhat from the strategy in
\cite{borg:scatter, ckstt:gwp, RV, tao:gwp radial, tvz-higher, thesis:art, Monica:thesis},
which primarily sought to use spacetime  estimates (especially those arising from Morawetz inequalities)
to exclude blowup scenarios; control of regularity
(that is, decay at very high or very low frequencies) only appeared later in the argument.
But in two dimensions, the Morawetz inequalities are unfavourable and we have to rely on the virial inequality instead
(see Proposition~\ref{eliminate-i}), which does not provide good long-term spacetime control.
Thus, we need to establish regularity even in the absence of such spacetime estimates.  In this regard,
our approach is more similar to papers such as \cite{merlekenig, nawa, tao:radialfocus, tao-cc}.  One reward
for making do with the virial identity is that it opens up the possibility of treating the focusing problem,
as was conclusively demonstrated by Kenig and Merle, \cite{merlekenig}.

At first glance, additional regularity may seem unreasonable since $u$ is
\emph{a priori} only known to have finite mass and dispersive
equations such as \eqref{nls} do not exhibit global smoothing
properties.  However, the solutions appearing in Theorem~\ref{comp}
are very special in nature, for example, they are almost periodic modulo symmetries.
The possibility of proving such regularity becomes still more plausible when one recalls
how Theorem~\ref{main-compact} is proved.  The solution $u$ is selected to have minimal mass
among all blowup solutions; this means that there can be no waste.  Adding high-frequency ripples
that do not directly contribute to blowup would constitute an example of such waste.  (The same
rationale explains why $u$ must concentrate in a single bubble both in space and in frequency.)
A further manifestation of this `eco-friendly' (that is, no waste) property is the absence
of a scattered wave at the endpoints of the lifespan $I$; more formally, we have the following
Duhamel formula, which plays an important role in proving additional regularity.

\begin{lemma}[{\cite[Section 6]{compact}}]\label{duhamel L}
Let $u$ be an almost periodic solution to \eqref{nls} on its maximal-lifespan $I$.  Then, for all $t\in I$,
\begin{equation}\label{duhamel}
\begin{aligned}
u(t)&=\lim_{T\nearrow\,\sup I}i\int_t^T e^{i(t-t')\Delta} F(u(t'))\,dt'\\
&=-\lim_{T\searrow\,\inf I}i\int_T^t e^{i(t-t')\Delta} F(u(t'))\,dt',
\end{aligned}
\end{equation}
as weak limits in $L_x^2$.
\end{lemma}

Additional regularity for the self-similar solution is proved in Section~\ref{ss-sec}.
The argument  relies on iterating various versions of the Strichartz inequality
(including a recent refinement of that inequality in the spherically symmetric case due to Shao \cite{Shuanglin}),
taking full advantage of the self-similarity to control the motion of mass between frequencies.
In the latter part of Section~\ref{ss-sec} we disprove the existence of the self-similar solution by noting that
$H_x^1$ solutions are global (see the discussion after Conjecture~\ref{conj}), while the self-similar solution is not.

For the remaining two cases, higher regularity is proved in Section~\ref{glob-sob-sec}.
The idea is to exploit the global existence, together with the almost periodicity modulo scaling,
by applying the Duhamel formula \eqref{duhamel} both in the past and in the future.
If done naively, neither of these Duhamel integrals will be absolutely convergent.
However, in Section~\ref{pseudo-sec} we describe the decomposition into incoming waves
(which we propagate backwards in time) and outgoing waves (which we propagate forward in time).
We can then take advantage of the radial symmetry, which concentrates the solution near the origin.
In this way, we obtain convergent integrals and regularity will then be obtained by a simple iteration argument.

In Section~\ref{hilo cascade-sec}, we use the additional regularity together with the conservation of energy
to preclude the double high-to-low frequency cascade.  In Section~\ref{soliton-sec}, we disprove the existence
of soliton-like solutions using a truncated virial identity in much the same manner as \cite{merlekenig}.

Corollary~\ref{cor:conc} is proved in Section~\ref{cor-sec}.

\subsection*{Acknowledgements}
We are grateful to Shuanglin Shao for access to preliminary drafts of his work.

R.~K. was supported by NSF grant DMS-0401277 and a Sloan Foundation Fellowship.   He is also
grateful to the Institute for Advanced Study (Princeton) for its hospitality.  T.~T. was supported by a grant from the
MacArthur Foundation.  M.~V. was supported under NSF grant DMS-0111298 and as a Liftoff Fellow of the Clay Mathematics Institute.

Any opinions, findings and conclusions or recommendations expressed are those of the authors and do not
reflect the views of the National Science Foundation.

\section{Notation and linear estimates}

This section contains the basic linear estimates we use repeatedly in the paper.

\subsection{Some notation}
We use $X \lesssim Y$ or $Y \gtrsim X$ whenever $X \leq CY$ for some constant $C>0$.  We use
$O(Y)$ to denote any quantity $X$ such that $|X| \lesssim Y$.  We use the notation $X \sim Y$
whenever $X \lesssim Y \lesssim X$.  The fact that these constants depend upon the dimension $d$
will be suppressed.  If $C$ depends upon some additional parameters, we will indicate this with subscripts;
for example, $X \lesssim_u Y$ denotes the assertion that $X \leq C_u Y$ for some $C_u$ depending on $u$;
similarly for $X \sim_u Y$, $X = O_u(Y)$, etc.

We use the `Japanese bracket' convention $\langle x \rangle := (1 +|x|^2)^{1/2}$.

We use $L^q_t L^r_{x}$ to denote the Banach space with norm
$$ \| u \|_{L^q_t L^r_x(\R \times \R^d)} := \Bigl(\int_\R \Bigl(\int_{\R^d} |u(t,x)|^r\ dx\Bigr)^{q/r}\ dt\Bigr)^{1/q},$$
with the usual modifications when $q$ or $r$ are equal to infinity, or when the domain $\R \times \R^d$
is replaced by a smaller region of spacetime such as $I \times \R^d$.  When $q=r$ we abbreviate
$L^q_t L^q_x$ as $L^q_{t,x}$.

\subsection{Basic harmonic analysis}
Let $\varphi(\xi)$ be a radial bump function supported in the ball $\{ \xi \in \R^d: |\xi| \leq \tfrac {11}{10} \}$ and equal to $1$
on the ball $\{ \xi \in \R^d: |\xi| \leq 1 \}$.  For each number $N > 0$, we define the Fourier multipliers
\begin{align*}
\widehat{P_{\leq N} f}(\xi) &:= \varphi(\xi/N) \hat f(\xi)\\
\widehat{P_{> N} f}(\xi) &:= (1 - \varphi(\xi/N)) \hat f(\xi)\\
\widehat{P_N f}(\xi) &:= \psi(\xi/N)\hat f(\xi) := (\varphi(\xi/N) - \varphi(2\xi/N)) \hat f(\xi).
\end{align*}
We similarly define $P_{<N}$ and $P_{\geq N}$.  We also define
$$ P_{M < \cdot \leq N} := P_{\leq N} - P_{\leq M} = \sum_{M < N' \leq N} P_{N'}$$
whenever $M < N$.  We will usually use these multipliers when $M$ and $N$ are \emph{dyadic numbers}
(that is, of the form $2^n$ for some integer $n$); in particular, all summations over $N$ or $M$ are understood
to be over dyadic numbers.  Nevertheless, it will occasionally be convenient to allow $M$ and $N$ to not be a power of $2$.
Note that $P_N$ is not truly a projection; to get around this, we will occasionally need to use fattened Littlewood-Paley operators:
\begin{equation}\label{PMtilde}
\tilde P_N := P_{N/2} + P_N +P_{2N}.
\end{equation}
These obey $P_N \tilde P_N = \tilde P_N P_N= P_N$.

As with all Fourier multipliers, the Littlewood-Paley operators commute
with the propagator $e^{it\Delta}$, as well as with differential operators such as $i\partial_t + \Delta$.
We will use basic properties of these operators many many times, including

\begin{lemma}[Bernstein estimates]\label{Bernstein}
For $1 \leq p \leq q \leq \infty$,
\begin{align*}
\|P_{\leq N} f\|_{L^q_x(\R^d)} &\lesssim N^{\frac{d}{p}-\frac{d}{q}} \|P_{\leq N} f\|_{L^p_x(\R^d)},\\
\|P_N f\|_{L^q_x(\R^d)} &\lesssim N^{\frac{d}{p}-\frac{d}{q}} \| P_N f\|_{L^p_x(\R^d)}.
\end{align*}
\end{lemma}

One of the ways in which we exploit the spherical symmetry assumption in this paper is by using the stronger spatial
decay as $|x| \to \infty$ that spherically symmetric functions enjoy.  One expression of this
is the following `endpoint' radial Sobolev embedding estimate.

\begin{lemma}[Radial Sobolev embedding]\label{radial embedding}
For spherically symmetric $f: \R^d \to \C$,
\begin{align}\label{endpoint rad sob}
\bigl\||x|^{\frac {d-1}2} f_N \bigr\|_{L^\infty_x} \lesssim N^{\frac 12}\|f_N\|_{L^2_x}
\end{align}
for any frequency $N\in 2^{\Z}$.
\end{lemma}

\begin{proof}
The claim follows from a more general inequality obeyed by radial functions,
\begin{equation}\label{radial GN}
\bigl\| |x|^{\frac {d-1}2} f \bigr\|_{L^\infty_x}^2 \lesssim \|f\|_{L^2_x} \|\nabla f\|_{L^2_x}.
\end{equation}
This in turn, can be deduced as follows.  By the Fundamental Theorem of Calculus and H\"older's inequality,
\begin{align*}
|g(r)|^2 &= \Bigl| \int_r^\infty g(\rho) g'(\rho) \,d\rho \Bigr|
        \leq r^{1-d} \Bigl| \int_r^\infty g(\rho) g'(\rho) \rho^{d-1} d\rho \Bigr| \\
&\leq r^{1-d} \| g \|_{L^2(\rho^{d-1}\,d\rho)}\| g' \|_{L^2(\rho^{d-1}\,d\rho)}
\end{align*}
for any Schwartz function $g:[0,\infty)\to\C$.  To obtain \eqref{radial GN}, we apply this to $g(|x|)=f(x)$.
\end{proof}

In a series of papers culminating in \cite{ChristKiselev}, Christ and Kiselev proved that under
certain hypotheses, truncations of bounded integral operators are themselves bounded.  We will
only use the simplest form of their results (the proof is based on a
Whitney decomposition of the half space $t<s$):

\begin{lemma}[Christ--Kiselev, \cite{ChristKiselev}]\label{L:ChristKiselev}
For each $s,t\in\R$, let $T(t,s):L^r_x\to L^{\tilde r}_x$ denote a linear operator.  If
\begin{equation*}
\Bigl\| \int_{\R} T(t,s) F(s)\,ds \Bigr\|_{L^q_t L^r_x} \lesssim \bigl\| F \bigr\|_{L^p_t L^{\tilde r}_x}
\end{equation*}
for some $q>p$, then
\begin{equation*}
\Bigl\| \int_{-\infty}^t T(t,s)F(s)\,ds \Bigr\|_{L^q_t L^r_x} \lesssim_{p,q} \bigl\| F \bigr\|_{L^p_t L^{\tilde r}_x}.
\end{equation*}
\end{lemma}

This lemma is handy because the Duhamel formula \eqref{old duhamel} comes with the constraint $s<t$,
but it is much easier to prove estimates after discarding it.

\subsection{Strichartz estimates}

Naturally, everything that we do for the nonlinear Schr\"odinger equation builds on basic properties of the
linear propagator $e^{it\Delta}$.

From the explicit formula
$$ e^{it\Delta} f(x) = \frac{1}{(4\pi i t)^{d/2}} \int_{\R^d} e^{i|x-y|^2/4t} f(y)\, dy,$$
we deduce the standard dispersive inequality
\begin{equation}\label{dispersive}
\| e^{it\Delta} f \|_{L^\infty(\R^d)} \lesssim \frac{1}{|t|^{d/2}} \| f \|_{L^1(\R^d)}
\end{equation}
for all $t\neq 0$.

Finer bounds on the (frequency localized) linear propagator can be derived using stationary phase:

\begin{lemma}[Kernel estimates]\label{lr:propag est L}
For any $m\geq 0$, the kernel of the linear propagator obeys the following estimates:
\begin{equation}\label{lr:large times}
\bigr| ( P_N e^{it\Delta} )(x,y) \bigl| \lesssim_m \begin{cases}
    |t|^{-d/2} & :\ |x-y| \sim N|t| \\[0.5ex]
    \displaystyle \frac{N^d}{|N^2t|^{m} \langle N|x-y| \rangle^{m} } \quad &:\ \text{otherwise}
    \end{cases}
\end{equation}
for $|t|\geq N^{-2}$ and
\begin{equation}\label{lr:small times}
\bigr| ( P_N e^{it\Delta} )(x,y) \bigl| \lesssim_m N^d \bigl\langle N|x-y| \bigr\rangle^{-m}
\end{equation}
for $|t|\leq  N^{-2}$.
\end{lemma}

We also record the following standard Strichartz estimate:

\begin{lemma}[Strichartz]\label{L:strichartz}  Let $I$ be an interval, let $t_0 \in I$, and let $u_0 \in L^2_x(\R^d)$
and $f \in L^{2(d+2)/(d+4)}_{t,x}(I \times \R^d)$.  Then, the function $u$ defined by
$$ u(t) := e^{i(t-t_0)\Delta} u_0 - i \int_{t_0}^t e^{i(t-t')\Delta} f(t')\ dt'$$
obeys the estimate
$$
\|u \|_{C^0_t L^2_x(I \times \R^d)} + \| u \|_{L^{2(d+2)/d}_{t,x}(I \times \R^d)}
    \lesssim \| u_0 \|_{L^2_x(\R^d)} + \|f\|_{L^{2(d+2)/(d+4)}_{t,x}(I \times \R^d)}.
$$
Moreover, if $d=2$ and both $u$ and $f$ are spherically symmetric,
$$ \| u \|_{L^2_t L^\infty_x(I \times \R^2)}
    \lesssim \| u_0 \|_{L^2_x(\R^2)} + \|f\|_{L^{4/3}_{t,x}(I \times \R^2)}.
$$
\end{lemma}

\begin{proof}
See, for example, \cite{gv:Strichartz, Strichartz} for the first inequality and
\cite{stefanov, tao:spherical} for the endpoint case in dimension $d=2$.
In \cite{montg-s} it is shown that the $L^2_t L^\infty_x$ endpoint fails without the
assumption of spherical symmetry.
\end{proof}

\begin{remark}
There are of course more Strichartz estimates available than the ones listed here, but these (together with their adjoints)
are the only ones we will need.  By the Duhamel formula \eqref{old duhamel}, one can of course write
$u(t_0)$ in place of $u_0$ and $iu_t + \Delta u$ in place of $f$.
\end{remark}

We will also need three variants of the Strichartz inequality.  We will only record the two-dimensional case of these.
The first is a bilinear variant, which will be useful for controlling interactions between widely separated frequencies:

\begin{lemma}[Bilinear Strichartz]\label{L:bilinear strichartz}
For any spacetime slab $I \times \R^2$, any $t_0 \in I$, and any $M, N > 0$, we have
\begin{align*}
\|( P_{\geq N} u) (P_{\leq M} v) \|_{L^2_{t,x}(I \times \R^2)}
& \lesssim \bigl( \tfrac{M}N \bigr)^\frac12 \bigl(\|P_{\geq N} u(t_0)\|_{L^2} + \| (i \partial_t + \Delta)
     P_{\geq N} u \|_{L^{\frac 43}_{t,x}(I\times\R^2)}\bigr)\\
&\qquad \times \bigl(\|P_{\leq M} v(t_0)\|_{L^2} + \| (i\partial_t + \Delta)P_{\leq M} v\|_{L^{\frac 43}_{t,x}(I\times\R^2)}\bigr),
\end{align*}
for all functions $u, v$ on $I$.
\end{lemma}

\begin{proof}
See \cite[Lemma 2.5]{Monica:thesis}, which builds on earlier versions in \cite{borg:book, ckstt:gwp}.
More general estimates hold for other dimensions and other exponents, but we will not need those here.
\end{proof}

Next, we observe a weighted endpoint Strichartz estimate, which exploits the spherical symmetry heavily
in order to obtain spatial decay. It is very useful in regions of space far from the origin $x=0$.

\begin{lemma}[Weighted endpoint Strichartz]\label{L:wes}  Let $I$ be an interval, let $t_0 \in I$, and let $u_0 \in L^2_x(\R^2)$
and $f \in L^{4/3}_{t,x}(I \times \R^2)$ be spherically symmetric.  Then, the function $u$ defined by
$$ u(t) := e^{i(t-t_0)\Delta} u_0 - i \int_{t_0}^t e^{i(t-t')\Delta} f(t')\ dt'$$
obeys the estimate
$$ \bigl\| |x|^{1/2} u \bigr\|_{L^4_t L^\infty_x(I \times \R^2)}
\lesssim \| u_0 \|_{L^2_x(\R^2)} + \|f\|_{L^{4/3}_{t,x}(I \times \R^2)}.$$
\end{lemma}

\begin{proof}
As in the usual proof of Strichartz inequality, the main ingredients are a dispersive estimate and fractional integration.

Using Lemma \ref{L:strichartz} and the Christ-Kiselev lemma, we see that it suffices to prove the free estimate
\begin{align}\label{wait S}
 \bigl\| |x|^{1/2} e^{it\Delta} u_0 \bigr\|_{L^4_t L^\infty_x(\R \times \R^2)} \lesssim \| u_0 \|_{L^2_x(\R^2)}
\end{align}
for spherically symmetric $u_0$. By the $TT^*$ method, this is equivalent to
$$
\Bigl \| \int_\R |x|^{1/2} e^{i(t-t')\Delta} |x|^{1/2} f(t')\ dt' \Bigr\|_{L^4_t L^\infty_x(\R \times \R^2)}
\lesssim \| f \|_{L^{4/3}_t L^1_x(\R \times \R^2)}.
$$
By the Hardy-Littlewood-Sobolev theorem of fractional integration, it thus suffices to prove the following
weighted dispersive inequality:
$$ \bigl\| |x|^{1/2} e^{i(t-t')\Delta} |x|^{1/2} f(t') \bigr\|_{L^\infty_x(\R^2)} \lesssim |t-t'|^{-1/2} \|f(t')\|_{L^1_x(\R^2)}.$$
In order to achieve this, we note that for radial functions the propagator has integral kernel
$$
e^{it\Delta}(x,y) = \frac1{4\pi it}\int_0^{2\pi} \exp\Bigl\{\frac{i|x|^2+i|y|^2-2i|x||y|\cos(\theta)}{4t}\Bigr\}
        \,\frac{d\theta}{2\pi}.
$$
Thus, a standard application of stationary phase (or the behaviour of the $J_0$ Bessel function) reveals that
$|e^{it \Delta }(x,y)| \lesssim ( |x||y| |t|)^{-1/2}$.  In particular,
$$
\bigl\| |x|^{\frac12}e^{it\Delta}|x|^{\frac12} \bigr\|_{L^1_x\to L^\infty_x} \lesssim |t|^{-1/2}
$$
and the claim follows.
\end{proof}

We will rely crucially on a slightly different type of improvement to the Strichartz inequality
in the spherically symmetric case due to Shao \cite{Shuanglin}, which improves the spacetime decay of the solution
after localizing in frequency:

\begin{lemma}[Shao's Strichartz Estimate, {\cite[Corollary~6.2]{Shuanglin}}]\label{L:Shuanglin}
For $f\in L^2_\rad(\R^2)$ we have
\begin{equation}\label{E:Shuanglin}
\| P_N e^{it\Delta} f \|_{L^q_{t,x}(\R\times\R^2)} \lesssim_q N^{1-\frac4q} \| f \|_{L^2_x(\R^2)},
\end{equation}
provided $q>\frac{10}3$.
\end{lemma}

The key point for us is that $q$ can go below $4$, which is the exponent given by Lemma~\ref{L:strichartz}.
The Knapp counterexample (a wave-packet whose momentum is concentrated in a single direction) shows that
such an improvement is not possible without the radial assumption.  One can combine this estimate with the usual
Strichartz estimates by the standard Christ-Kiselev trick to obtain similar estimates for solutions
to an inhomogeneous Schr\"odinger equation, but we will not need those estimates here.

Our last lemma is essentially Bourgain's concentration argument from \cite[\S2--3]{bourg.2d}.
It may be interpreted as an inverse Strichartz inequality.
It roughly says that the Strichartz norm cannot be large without there being a bubble of concentration in spacetime.
While results of this genre constitute an important precursor to the concentration compactness technique, our
only application of it will be in the proof of Corollary~\ref{cor:conc}; see Section~\ref{cor-sec}.

\begin{lemma}\label{L:B conc}
Given $\phi\in L^2_x(\R^2)$ and $\eta>0$, there exists $C=C(M(\phi), \eta)$ so that if
$$
\int_I \int_{\R^2} \bigl|e^{it\Delta} \phi \bigr|^4 \,dx\,dt \geq \eta
$$
for some interval $I\subseteq \R$, then there exist $x_0, \xi_0\in\R^2$ and $J\subseteq I$ so that
\begin{equation}\label{B conc}
\int_{|x-x_0-2t\xi_0|\leq C |J|^{1/2}} \bigl|e^{it\Delta} \phi \bigr|^2 \,dx \geq C^{-1} \quad\text{for all} \quad t\in J.
\end{equation}
(Notice that $C$ does not depend on $I$ or $J$.)  If $\phi$ is spherically symmetric, then we may take $x_0=\xi_0=0$.
\end{lemma}

\begin{proof}
As noted, a proof can be found in \cite{bourg.2d}.  For the convenience of the reader, we will give a short proof
in the radial case using the tools described in this section.

Making a Littlewood--Paley decomposition of $u$ and then applying H\"older followed by
the Strichartz and bilinear Strichartz inequalities, we find
\begin{align*}
\eta &\lesssim \sum_{N\geq M} \|e^{it\Delta} \phi_N e^{it\Delta} \phi_M\|_{L^2_{t,x}(I\times\R^2)}
        \|e^{it\Delta} \phi_N\|_{L^4_{t,x}(I\times\R^2)} \|e^{it\Delta} \phi_M\|_{L^4_{t,x}(I\times\R^2)} \\
&\lesssim \sum_{N\geq M} \Bigl(\frac{M}N\Bigr)^{1/2} \|\phi_N\|_{L^2_x}^2 \|\phi_M\|_{L^2_x}
        \|e^{it\Delta} \phi_M\|_{L^4_{t,x}(I\times\R^2)} \\
&\lesssim \|\phi\|_{L^2_x}^3 \sup_M \|e^{it\Delta} \phi_M\|_{L^4_{t,x}(I\times\R^2)}.
\end{align*}
Therefore,
\begin{equation*}
\sup_M \|e^{it\Delta} \phi_M\|_{L^4_{t,x}(I\times\R^2)} \gtrsim \eta \|\phi\|_{L^2_x}^{-3}.
\end{equation*}
On the other hand, by Bernstein's inequality (Lemma~\ref{Bernstein}),
$$
\int_I \int_{\R^2} \bigl|e^{it\Delta} \phi_{M} \bigr|^4 \,dx\,dt \lesssim |I| M^2 \|\phi\|_{L^2_x}^4.
$$
Combining these two, we see that there is $M\gtrsim_{M(\phi)} \eta^{1/2} |I|^{-1/2}$ so that
\begin{equation}\label{something}
\|e^{it\Delta} \phi_{M}\|_{L_{t,x}^4(I\times\R^2)}\gtrsim_{M(\phi)} \eta.
\end{equation}

By Lemma~\ref{L:Shuanglin} we obtain the upper bound
$$
\|e^{it\Delta} \phi_{M}\|_{L_{t,x}^{7/2}(I\times\R^2)}\lesssim M^{-1/7}\|\phi\|_{L_x^2},
$$
which combined with \eqref{something} and H\"older's inequality yields
$$
\|e^{it\Delta} \phi_{M}\|_{L_{t,x}^\infty(I\times\R^2)}\gtrsim_{M(\phi)} \eta^{8} M.
$$
Thus, there exist $t_0\in I$ and $x_0\in\R^2$ so that
$$
\bigl| [e^{it_0\Delta}\phi_M](x_0) \bigr| \gtrsim_{M(\phi)} \eta^{8} M.
$$
Using the fact that the kernel of $P_M$ is concentrated in the ball of radius $M^{-1}$ and that it
has $L^2_x$ norm comparable to $M$, we may conclude that
$$
\int_{|x-x_0|\lesssim M^{-1}} \bigl| \tilde P_M e^{it_0\Delta}\phi (x) \bigr|^2\,dx \gtrsim_{M(\phi),\eta}  1.
$$
Here $\tilde P_M=P_{M/2} + P_M +P_{2M}$ as in \eqref{PMtilde}. By enlarging the ball a little
and using \eqref{lr:small times} we obtain
$$
\int_{|x-x_0|\lesssim M^{-1}} \bigl| e^{it\Delta}\phi (x) \bigr|^2\,dx \gtrsim_{M(\phi),\eta}  1
$$
for all $|t-t_0|\leq M^{-2}$.  Let $J:=\{t\in I:\, |t-t_0|\leq M^{-2}\}$. To obtain \eqref{B conc} with $\xi_0=0$,
we simply note that because of our lower bound on $M$, the length of $J$ obeys $|J|\gtrsim_{M(\phi),\eta} M^{-2}$.

As $\phi$ is spherically symmetric, we must have $|x_0|\lesssim_{M(\phi)}|J|^{1/2}$; otherwise, we may find a large number of
disjoint balls each containing $\eta$ amount of mass, which contradicts the finiteness of the $L_x^2$-norm of $\phi$. Thus, it
is possible to set $x_0=0$ by simply enlarging $C$.
\end{proof}

\section{Almost periodic solutions}\label{apsec}

In this section, we record basic facts about the frequency scale function $N(t)$ that will be needed in the proof of
Theorem~\ref{comp} in the next section.  Much of the theory here is implicit in \cite{keraani,compact}.
Throughout this section, we fix the dimension $d$ and the sign $\mu$.  The results do not rely on spherical symmetry.

\begin{lemma}[Quasi-uniqueness of $N$]\label{uniq}
Let $u$ be a non-zero solution to \eqref{nls} with lifespan $I$ that is almost periodic modulo $G$
with frequency scale function $N: I \to \R^+$ and compactness modulus function $C: \R^+ \to \R^+$,
and also almost periodic modulo $G$ with frequency scale function $N': I \to \R^+$ and
compactness modulus function $C': \R^+ \to \R^+$.  Then we have
$$ N(t) \sim_{u,C,C'} N'(t)$$
for all $t \in I$.
\end{lemma}

\begin{proof}
By symmetry, it suffices to establish the bound $N'(t) \lesssim_{u,C,C'} N(t)$.  We write $x'(t)$ for the spatial
center associated to $N'$ and $C'$.  Similarly, $\xi(t)$ is the frequency center associated to $N$ and $C$.

Fix $t$ and let $\eta > 0$ to be chosen later. By Definition~\ref{apdef} we have
$$ \int_{|x-x'(t)| \geq C'(\eta)/N'(t)} |u(t,x)|^2\ dx \leq \eta$$
and
$$ \int_{|\xi-\xi(t)| \geq C(\eta) N(t)} |\hat u(t,\xi)|^2\ d\xi \leq \eta.$$
We split $u(t,x) = u_1(t,x) + u_2(t,x)$, where $u_1(t,x) := u(t,x) 1_{|x-x'(t)|\geq C'(\eta)/N'(t)}$ and
$u_2(t,x) := u(t,x) 1_{|x-x'(t)|< C'(\eta)/N'(t)}$.  Then, by Plancherel's theorem we have
\begin{align}\label{u1}
\int_{\R^d} |\hat u_1(t,\xi)|^2\ d\xi \lesssim \eta,
\end{align}
while from Cauchy-Schwarz we have
$$ \sup_{\xi \in \R^d} |\hat u_2(t,\xi)|^2 \lesssim_{\eta,C'} M(u) N'(t)^{-d}.$$
Integrating the last inequality over the ball $|\xi-\xi(t)| \leq C(\eta) N(t)$ and using \eqref{u1}, we conclude that
$$ \int_{\R^d} |\hat u(t,\xi)|^2\ d\xi \lesssim \eta + O_{\eta,C,C'}(M(u) N(t)^d N'(t)^{-d}).$$
Thus, by Plancherel and mass conservation,
$$ M(u) \lesssim \eta + O_{\eta,C,C'}(M(u) N(t)^d N'(t)^{-d}).$$
Choosing $\eta$ to be a small multiple of $M(u)$ (which is non-zero by hypothesis), we obtain the claim.
\end{proof}

\begin{lemma}[Quasi-continuous dependence of $N$ on $u$]\label{con}
Let $u^{(n)}$ be a sequence of solutions to \eqref{nls} with lifespans $I^{(n)}$, which are almost periodic modulo scaling
with frequency scale functions $N^{(n)}: I^{(n)} \to \R^+$ and compactness modulus function $C: \R^+ \to \R^+$, independent of $n$.
Suppose that $u^{(n)}$ converge locally uniformly to a non-zero solution $u$ to \eqref{nls}
with lifespan $I$.  Then $u$ is almost periodic modulo scaling with a frequency scale function $N: I \to \R^+$
and compactness modulus function $C$.  Furthermore, we have
\begin{equation}\label{ann}
N(t) \sim_{u,C} \liminf_{n \to \infty} N^{(n)}(t) \sim_{u,C} \limsup_{n \to \infty} N^{(n)}(t)
\end{equation}
for all $t \in I$.  Finally, if all $u^{(n)}$ are spherically symmetric, then $u$ is also.
\end{lemma}

\begin{proof} We first show that
\begin{equation}\label{ninfsup}
0 < \liminf_{n \to \infty} N^{(n)}(t) \leq \limsup_{n \to \infty} N^{(n)}(t) < \infty
\end{equation}
for all $t \in I$.  Indeed, if one of these inequalities failed for some $t$, then (by passing to a subsequence if necessary)
$N^{(n)}(t)$ would converge to zero or to infinity as $n \to \infty$.  Thus, by Definition~\ref{apdef},
$u^{(n)}(t)$ would converge weakly to zero, and hence, by the local uniform convergence, would converge strongly to zero.
But this contradicts the hypothesis that $u$ is not identically zero.  This establishes \eqref{ninfsup}.

From \eqref{ninfsup}, we see that for each $t \in I$ the sequence $N^{(n)}(t)$ has at least one limit point $N(t)$.
Thus, using the local uniform convergence we easily verify that $u$ is almost periodic modulo scaling with
frequency scale function $N$ and compactness modulus function $C$.  It is also clear that if all $u^{(n)}$ are spherically symmetric,
then $u$ is also.

It remains to establish \eqref{ann}, which we prove by contradiction.  Suppose it fails.
Then given any $A = A_{u}$, there exists a $t \in I$ for which $N^{(n)}(t)$ has at least two limit points which
are separated by a ratio of at least $A$, and so $u$ has two frequency scale functions
with compactness modulus function $C$ which are separated by this ratio.  But this contradicts
Lemma~\ref{uniq} for $A$ large enough depending on $u$.  Hence \eqref{ann} holds.
\end{proof}

The following claim is easily verified.

\begin{lemma}[Symmetries of almost periodic solutions]\label{symap}
Let $u$ be a solution to \eqref{nls} with lifespan $I$ which is almost periodic modulo $G$ (respectively $G_{\rad}$)
with frequency scale function $N: I \to \R^+$ and compactness modulus function $C: \R^+ \to \R^+$.
Let the notation be as in Lemma \ref{sym}.  Then
\begin{CI}
\item The time reversal $\tilde u$ is almost periodic modulo $G$ (respectively $G_{\rad}$) with frequency scale function
$\tilde N(t) := N(-t)$ and compactness modulus function $C$.
\item For any $t_0 \in \R$, the time translation $u_{t_0}$ is almost periodic modulo $G$ (respectively $G_\rad$)
with frequency scale function $N_{t_0}(t) := N(t+t_0)$ and compactness modulus function $C$.
\item For any $\lambda > 0$, the rescaled solution $T_{g_{0,0,0,\lambda}} u$ is almost periodic modulo $G$ (respectively $G_\rad$)
with frequency scale function $N^{[\lambda]}(t) := N( t/\lambda^2 ) / \lambda$ and compactness modulus function $C$.
\end{CI}
\end{lemma}

There is an analogous claim for the pseudoconformal transformation, but it is more complicated and we will not need it here.
The difference is that the compactness modulus function $C$ changes; the closer $0$ is to $I$, the worse it becomes.

In the non-radial case one can also easily track the changes in the functions $\xi(t)$ and $x(t)$, as well as the actions
of more general elements of $G$, but we will not need to do so here.

Next, we state a useful compactness result.  First, a definition:

\begin{definition}[Normalised solution]
Let $u$ be a solution to \eqref{nls}, which is almost periodic modulo $G$ with frequency scale function $N$,
position center function $x$, and frequency center function $\xi$.  We say that $u$ is \emph{normalised} if the lifespan $I$
contains zero and
$$ N(0) = 1, \quad x(0) = \xi(0) = 0.$$
More generally, we can define the \emph{normalisation} of a solution $u$ at a time $t_0$ in its lifespan $I$ to be
\begin{equation}\label{untn}
u^{[t_0]} := T_{g_{0,-\xi(t_0)/N(t_0), -x(t_0)N(t_0),N(t_0)}}( u_{t_0} ).
\end{equation}
Observe that $u^{[t_0]}$ is a normalised solution which is almost periodic modulo $G$ and has lifespan
$$ I^{[t_0]} := \{ s \in \R: t_0 + s/N(t_0)^2 \in I \}$$
(so, in particular, $0 \in I^{[t_0]}$).  It has frequency scale function
$$ N^{[t_0]}(s) := N( t_0 + s/N(t_0)^2 ) / N(t_0)$$
and the same compactness modulus function as $u$. Furthermore, if $u$ is maximal-lifespan then so is $u^{[t_0]}$,
and if $u$ is spherically symmetric and almost periodic modulo $G_\rad$ then so is $u^{[t_0]}$.
\end{definition}

\begin{lemma}[Compactness of almost periodic solutions]\label{laps}
Let $u^{(n)}$ be a sequence of normalised maximal-lifespan solutions to \eqref{nls} with lifespans $I^{(n)}\ni 0$,
which are almost periodic modulo $G$ with frequency scale functions $N^{(n)}: I^{(n)} \to \R^+$ and a uniform compactness
modulus function $C: \R^+ \to \R^+$.  Assume that we also have a uniform mass bound
\begin{equation}\label{mb}
0 < \inf_n M(u^{(n)}) \leq \sup_n M(u^{(n)}) < \infty.
\end{equation}
Then, after passing to a subsequence if necessary, there exists a non-zero maximal-lifespan solution $u$ to \eqref{nls}
with lifespan $I \ni 0$ that is almost periodic modulo $G$, such that $u^{(n)}$ converge locally uniformly to $u$.
Moreover, if all $u^{(n)}$ are spherically symmetric and almost periodic modulo $G_\rad$, then $u$ is also.
\end{lemma}

\begin{proof}  By hypothesis and Definition \ref{apdef}, we see that for every $\eps > 0$ there exists $R > 0$ such that
$$ \int_{|x| \geq R} |u^{(n)}(0,x)|^2\ dx \leq \eps$$
and
$$ \int_{|\xi| \geq R} |\widehat{u^{(n)}}(0,\xi)|^2\ d\xi \leq \eps$$
for all $n$.  From this, \eqref{mb}, and the Ascoli--Arzela Theorem, we see that the sequence $u^{(n)}(0)$
is precompact in the strong topology of $L^2_x(\R^d)$.  Thus, by passing to a subsequence if necessary,
we can find $u_0 \in L^2_x(\R^d)$ such that $u^{(n)}(0)$ converge strongly to $u_0$ in $L^2_x(\R^d)$.
From \eqref{mb} we see that $u_0$ is not identically zero.

Now let $u$ be the maximal Cauchy development of $u_0$ from time $0$, with lifespan $I$.
By Theorem~\ref{local}, $u^{(n)}$ converge locally uniformly to $u$.  The remaining claims now follow from Lemma~\ref{con}.
\end{proof}

To illustrate how the above theory is used, let us now record some simple consequences that will be useful in the sequel.

\begin{corollary}[Local constancy of $N$]\label{qc} Let $u$ be a non-zero maximal-lifespan solution to \eqref{nls}
with lifespan $I$ that is almost periodic modulo $G$ with frequency scale function $N: I \to \R^+$.
Then there exists a small number $\delta$, depending on $u$, such that for every $t_0 \in I$ we have
\begin{equation}\label{nblow}
\bigl[t_0 - \delta N(t_0)^{-2}, t_0 + \delta N(t_0)^{-2}\bigr] \subset I
\end{equation}
and
\begin{equation}\label{nt}
N(t) \sim_u N(t_0)
\end{equation}
whenever $|t-t_0| \leq \delta N(t_0)^{-2}$.
\end{corollary}

\begin{proof}
Let us first establish \eqref{nblow}.  We argue by contradiction.  Assume \eqref{nblow} fails.  Then, there exist
sequences $t_n \in I$ and $\delta_n \to 0$ such that $t_n + \delta_n N(t_n)^{-2} \not \in I$ for all $n$.
Define the normalisations $u^{[t_n]}$ of $u$ from time $t_n$ by \eqref{untn}.  Then, $u^{[t_n]}$ are
maximal-lifespan normalised solutions whose lifespans $I^{[t_n]}$ contain $0$ but not $\delta_n$;
they are also almost periodic modulo $G$ with frequency scale functions
\begin{equation}\label{nnsdef}
N^{[t_n]}(s) := N( t_n + s N(t_n)^{-2} ) / N(t_n)
\end{equation}
and the same compactness modulus function $C$ as $u$. Applying Lemma~\ref{laps} (and passing to a subsequence if necessary),
we conclude that $u^{[t_n]}$ converge locally uniformly to a maximal-lifespan solution $v$ with some lifespan $J \ni 0$.
By Theorem~\ref{local}, $J$ is open and so contains $\delta_n$ for all sufficiently large $n$.  This contradicts the
local uniform convergence as, by hypothesis, $\delta_n$ does not belong to $I^{[t_n]}$.  Hence \eqref{nblow} holds.

We now show \eqref{nt}.  Again, we argue by contradiction, shrinking $\delta$ if necessary.
Assume \eqref{nt} fails no matter how small one selects $\delta$.  Then, one can find sequences $t_n, t'_n \in I$
such that $s_n := (t'_n - t_n) N(t_n)^2 \to 0$ but $N(t'_n)/N(t_n)$ converge to either zero or infinity.
If we define $u^{[t_n]}$ and $N^{[t_n]}$ as before and  apply Lemma~\ref{laps} (passing to a subsequence if necessary),
we see once again that $u^{[t_n]}$ converge locally uniformly to a maximal-lifespan solution $v$ with some open lifespan $J \ni 0$.
But then $N^{[t_n]}(s_n)$ converge to either zero or infinity and thus, by Definition~\ref{apdef},
$u^{[t_n]}(s_n)$ are converging weakly to zero. On the other hand, since $s_n$ converge to zero and $u^{[t_n]}$ are locally
uniformly convergent to $v \in C^0_{t,\loc} L^2_x(J \times \R^d)$, we may conclude that $u^{[t_n]}(s_n)$ converge
strongly to $v(0)$ in $L^2_x(\R^2)$.  Thus $v(0) = 0$ and $M(u^{[t_n]})$ converge to $M(v)=0$.
But since $M(u^{(n)}) = M(u)$, we see that $u$ vanishes identically, a contradiction.  Thus \eqref{nt} holds.
\end{proof}

\begin{corollary}[Blowup criterion]\label{blow}
Let $u$ be a non-zero maximal-lifespan solution to \eqref{nls} with lifespan $I$ that is almost periodic modulo $G$
with frequency scale function $N: I \to \R^+$.  If $T$ is any finite endpoint of $I$, then
$N(t) \gtrsim_u |T-t|^{-1/2}$; in particular, $\lim_{t\to T} N(t)=\infty$.
\end{corollary}

\begin{proof}
This is immediate from \eqref{nblow}.
\end{proof}

\begin{lemma}[Local quasi-boundedness of $N$]\label{qb}
Let $u$ be a non-zero solution to \eqref{nls} with lifespan $I$ that is almost periodic modulo $G$
with frequency scale function $N: I \to \R^+$.  If $K$ is any compact subset of $I$, then
$$ 0 < \inf_{t \in K} N(t) \leq \sup_{t \in K} N(t) < \infty.$$
\end{lemma}

\begin{proof}
We only prove the first inequality; the argument for the last is similar.

We argue by contradiction.  Suppose that the first inequality fails.
Then, there exists a sequence $t_n \in K$ such that $\lim_{n \to \infty} N(t_n) = 0$ and hence, by Definition~\ref{apdef},
$u(t_n)$ converge weakly to zero.  Since $K$ is compact, we can assume $t_n$ converge to a limit $t_0 \in K$.
As $u\in C_t^0 L_x^2(K\times\R^d)$, we see that $u(t_n)$ converge strongly to $u(t_0)$.  Thus $u(t_0)$ must be zero,
contradicting the hypothesis.
\end{proof}

Finally, we establish a spacetime bound.

\begin{lemma}[Spacetime bound]\label{spacelemma}
Let $u$ be a non-zero solution to \eqref{nls} with lifespan $I$, which is almost periodic modulo $G$
with frequency scale function $N: I \to \R^+$.  If $J$ is any subinterval of $I$, then
\begin{equation}\label{uj-conv}
\int_J N(t)^2\,dt \lesssim_u \int_J \int_{\R^d} |u(t,x)|^{\frac{2(d+2)}{d}}\, dx\, dt\lesssim_u 1 + \int_J N(t)^2\,dt.
\end{equation}
\end{lemma}

\begin{proof}
We first prove
\begin{equation}\label{uj}
\int_J \int_{\R^d} |u(t,x)|^{\frac{2(d+2)}{d}}\, dx\, dt \lesssim_u 1 + \int_J N(t)^2\, dt.
\end{equation}
Let $0<\eta<1$ to be chosen momentarily and partition $J$ into subintervals $I_j$ so that
\begin{equation}\label{44 I_j defn}
\int_{I_j} N(t)^2 \,dt \leq \eta;
\end{equation}
this requires  at most $\eta^{-1}\times \text{RHS\eqref{uj}}$ many intervals.  For each $j$, we may choose
$t_j\in I_j$ so that
\begin{equation}\label{44 t_j}
N(t_j)^2 |I_j| \leq 2 \eta.
\end{equation}
By Strichartz inequality, we have the following estimates on the spacetime slab $I_j\times\R^d$
\begin{align*}
\|u\|_{L^{\frac{2(d+2)}{d}}_{t,x}}
&\lesssim \| e^{i(t-t_j)\Delta} u(t_j) \|_{L^{\frac{2(d+2)}{d}}_{t,x}}
        + \| u \|_{L^{\frac{2(d+2)}{d}}_{t,x}}^{\frac{d+4}{d}} \\
&\lesssim \| u_{\geq N_0}(t_j) \|_{L^2_x} + \| e^{i(t-t_j)\Delta} u_{\leq N_0}(t_j)\|_{L^{\frac{2(d+2)}{d}}_{t,x}}
        + \| u \|_{L^{\frac{2(d+2)}{d}}_{t,x}}^{\frac{d+4}{d}}\\
&\lesssim \| u_{\geq N_0}(t_j) \|_{L^2_x} + |I_j|^{\frac{d}{2(d+2)}} N_0^{\frac{d}{d+2}} \|u(t_j) \|_{L^2_x}
        + \| u \|_{L^{\frac{2(d+2)}{d}}_{t,x}}^{\frac{d+4}{d}};
\end{align*}
the last step here used H\"older's and Bernstein's inequalities. Choosing $N_0$ as a large multiple of $N(t_j)$
and using Definition~\ref{apdef}, one can make the first term as small as one wishes.  Subsequently,
choosing $\eta$ sufficiently small depending on $M(u)$ and invoking \eqref{44 t_j},
one may also render the second term arbitrarily small.  Thus, by the usual bootstrap argument we obtain
\begin{equation*}
\int_{I_j}\int_{\R^d} |u(t,x)|^\frac{2(d+2)}{d}\,dx\,dt \leq 1.
\end{equation*}
Using the bound on the number of intervals $I_j$, this leads to \eqref{uj}.

Now we prove
\begin{equation}\label{jaunt}
\int_J \int_{\R^d} |u(t,x)|^{\frac{2(d+2)}{d}}\, dx\, dt \gtrsim_u \int_J N(t)^2\, dt.
\end{equation}
Using Definition~\ref{apdef} and choosing $\eta$ sufficiently small depending on $M(u)$, we can guarantee that
\begin{align}\label{mass big}
\int_{|x-x(t)|\leq C(\eta)N(t)^{-1}} |u(t,x)|^2\,dx \gtrsim_u 1
\end{align}
for all $t\in J$; here $x(t)$ is the spatial center given to us by compactness.  On the other hand,
a simple application of H\"older's inequality yields
\begin{align*}
\int_{\R^d} |u(t,x)|^{\frac{2(d+2)}{d}}\,dx
\gtrsim_{u} \Bigl(\int_{|x-x(t)|\leq C(\eta)N(t)^{-1}} |u(t,x)|^2\Bigr)^{\frac{d+2}{d}} N(t)^2.
\end{align*}
Thus, using \eqref{mass big} and integrating over $J$ we derive \eqref{jaunt}.
\end{proof}

\begin{corollary}[Maximal-lifespan almost periodic solutions blow up]\label{blowup}
Let $u$ be a maximal-lifespan solution to \eqref{nls} which is almost periodic modulo $G$.
Then $u$ blows up both forward and backward in time.
\end{corollary}

\begin{proof}
At a finite endpoint, Corollary~\ref{blow} shows that the integral $\int_I N(t)^2\,dt$ diverges in any
neighbourhood of that endpoint.  Thus by Lemma~\ref{spacelemma} the spacetime norm also diverges,
which is the definition of blowup.

In the case of an infinite endpoint, we choose $t_0\in I$ and note that by \eqref{nblow},
$N(t)\gtrsim_u \langle t-t_0 \rangle^{-1/2}$.  The claim now follows as before.
\end{proof}

\section{Three enemies}\label{comp-sec}

We now prove Theorem~\ref{comp}.  The proof is unaffected by spherical symmetry, the dimension $d$, or the sign of $\mu$;
the essential ingredients are the scaling symmetry and the local constancy of $N(t)$.

Fix $d$ and $\mu$.  Invoking Theorem~\ref{main-compact}, we can find a maximal-lifespan
solution $v$ of some lifespan $J$ which is almost periodic modulo $G$ and blows up both
forward and backward in time; also, in the focusing case $\mu=-1$, we have $M(v) < M(Q)$.

Let $N_v(t)$ be the frequency scale function associated to $v$ as in Definition
\ref{apdef}, and let $C: \R^+ \to \R^+$ be the compactness modulus function. The
solution $v$ partially satisfies the conclusions of Theorem \ref{comp}, but we are not
necessarily in one of the three scenarios listed there.
To extract a solution $u$ with such additional properties, we will have to perform some further
manipulations primarily based on the scaling and time-translation symmetries.

For any $T \geq 0$, define the quantity
\begin{equation}\label{cdef}
\osc(T) := \inf_{t_0 \in J} \frac{\sup_{t \in J: |t-t_0| \leq T/N_v(t_0)^2} N_v(t)}{\inf_{t\in J: |t-t_0| \leq T/N_v(t_0)^2} N_v(t)}.
\end{equation}
Roughly speaking, this measures the least possible oscillation one can find
in $N_v$ on time intervals of normalised duration $T$.  This quantity is clearly
non-decreasing in $T$.  If $\osc(T)$ is bounded, we will be able to extract a soliton-like solution;
this is

\medskip

{\bf Case I:} $\lim_{T \to\infty} \osc(T) < \infty$.

In this case, we have arbitrarily long periods of stability for $N_v$.  More precisely, we can find
a finite number $A = A_v$, a sequence $t_n$ of times in $J$, and a sequence $T_n \to \infty$ such that
$$
\frac{\sup_{t \in J: |t-t_n| \leq T_n / N_v(t_n)^2} N_v(t)}{\inf_{t \in J: |t-t_n| \leq T_n / N_v(t_n)^2} N_v(t)} <  A
$$
for all $n$.  Note that this, together with Corollary~\ref{qc}, implies that
$$ [t_n - T_n / N_v(t_n)^2, t_n + T_n / N_v(t_n)^2] \subset J$$
and
$$ N_v(t) \sim_v N_v(t_n)$$
for all $t$ in this interval.

Now define the normalisations $v^{[t_n]}$ of $v$ at times $t_n$ as in \eqref{untn}.
Then $v^{[t_n]}$ is a maximal-lifespan normalised solution with lifespan
$$ J_n := \{ s \in \R: t_n + \frac{1}{N_v(t_n)^2} s \in J \} \supset [-T_n,T_n] $$
and mass $M(v)$.  It is almost periodic modulo scaling with frequency scale function
$$ N_{v^{[t_n]}}(s) := \frac{1}{N_v(t_n)} N_v(t_n + \frac{1}{N_v(t_n)^2} s)$$
and compactness modulus function $C$.  In particular, we see that
\begin{equation}\label{nvs}
N_{v^{[t_n]}}(s) \sim_v 1
\end{equation}
for all $s \in [-T_n,T_n]$.

We now apply Lemma~\ref{laps} and conclude (passing to a subsequence if necessary) that
$v^{[t_n]}$ converge locally uniformly to a maximal-lifespan solution $u$ with mass
$M(v)$ defined on an open interval $I$ containing $0$ and which is almost periodic modulo symmetries.
As $T_n\to\infty$, Lemma~\ref{con} and \eqref{nvs} imply that the frequency scale function $N: I \to \R^+$ of $u$
satisfies
$$ 0 < \inf_{t \in I} N(t) \leq \sup_{t \in I} N(t) < \infty.$$
In particular, by Corollary~\ref{blow}, $I$ has no finite endpoints and hence $I = \R$.
By modifying $C$ by a bounded amount we may now normalise $N \equiv 1$.  We
have thus constructed a soliton-like solution in the sense of Theorem~\ref{comp}.

\medskip

When $\osc(T)$ is unbounded, we must seek a solution belonging to one of the remaining two scenarios.
To distinguish between them, we introduce the quantity
$$ a(t_0) := \frac{ \inf_{t \in J: t \leq t_0} N_v(t) + \inf_{t \in J: t \geq t_0} N_v(t) }{N_v(t_0)}$$
for every $t_0 \in J$. This measures the extent to which $N_v(t)$ decays to zero on both sides of $t_0$.
Clearly, this quantity takes values in the interval $[0,2]$.

\medskip

{\bf Case II:} $\lim_{T\to \infty} \osc(T) = \infty$ and $\inf_{t_0 \in J} a(t_0) = 0$.

In this case, there are no long periods of stability but there are times about which there are arbitrarily large
cascades from high to low frequencies in both future and past directions.  This will allow us to extract a solution
with a double high-to-low frequency cascade as defined in Theorem~\ref{comp}.

As $\inf_{t_0 \in J} a(t_0) = 0$, there exists a sequence of times $t_n \in
J$ such that $a(t_n) \to 0$ as $n \to \infty$.  By the definition of $a$, we can also
find times $t_n^- < t_n < t_n^+$ with $t_n^-, t_n^+ \in J$ such that
$$ \frac{N_v(t_n^-)}{N_v(t_n)}\to0 \quad\text{and}\quad \frac{N_v(t_n^-)}{N_v(t_n)} \to 0.$$
Choose $t_n^- < t'_n < t_n^+$ so that
$$ N_v(t'_n) \sim \sup_{t_n^- \leq t \leq t_n^+} N_v(t);$$
then,
$$ \frac{N_v(t_n^-)}{N_v(t'_n)}\to0 \quad\text{and}\quad  \frac{N_v(t_n^-)}{N_v(t'_n)} \to 0.$$

We define the rescaled and translated times $s_n^- < 0 < s_n^+$ by
$$ s_n^\pm := N_v(t'_n)^2 (t^\pm_n - t'_n) $$
and the normalisations $v^{[t'_n]}$ at times $t'_n$ by \eqref{untn}.  These are normalised
maximal-lifespan solutions with lifespans containing $[s_n^-, s_n^+]$, which are almost periodic modulo $G$ with
frequency scale functions
\begin{equation}\label{nvns-def}
N_{v^{[t'_n]}}(s) := \frac{1}{N_v(t'_n)} N_v\bigl( t'_n + \frac{1}{N_v(t'_n)^2} s \bigr).
\end{equation}
By the way we chose $t_n'$, we see that
\begin{equation}\label{nvnt}
 N_{v^{[t'_n]}}(s) \lesssim 1
\end{equation}
for all $s_n^- \leq s \leq s_n^+$.  Moreover,
\begin{equation}\label{nvns}
 N_{v^{[t'_n]}}(s_n^\pm) \to 0 \quad\text{as}\quad n\to \infty
\end{equation}
for either choice of sign.

We now apply Lemma \ref{laps} and conclude (passing to a subsequence if necessary) that
$v^{[t'_n]}$ converge locally uniformly to a maximal-lifespan solution $u$ of mass
$M(v)$ defined on an open interval $I$ containing $0$, which is almost periodic modulo symmetries.

Let $N$ be a frequency scale function for $u$. From Lemma~\ref{qb} we see that $N(t)$
is bounded from below on any compact set $K\subset I$.  From this and Lemma~\ref{con} (and Lemma~\ref{uniq}), we
see that $N_{v^{[t'_n]}}(t)$ is also bounded from below, uniformly in $t \in K$, for
all sufficiently large $n$ (depending on $K$).  As a consequence of this and
\eqref{nvns}, we see that $s_n^-$ and $s_n^+$ cannot have any limit points in $K$;
thus $K \subset [s_n^-,s_n^+]$ for all sufficiently large $n$.  Therefore, $s_n^\pm$ converge to the endpoints of $I$.
Combining this with Lemma~\ref{con} and \eqref{nvnt}, we conclude that
\begin{equation}\label{nub}
\sup_{t \in I} N(t) < \infty.
\end{equation}
Lemma~\ref{blow} now implies that $I$ has no finite endpoints, that is, $I = \R$.

In order to prove that $u$ is a double high-to-low frequency cascade, we merely need to show that
\begin{equation}\label{limi}
 \liminf_{t \to +\infty} N(t) = \liminf_{t \to -\infty} N(t) = 0.
\end{equation}
By time reversal symmetry, it suffices to establish that $\liminf_{t \to
+\infty} N(t) = 0$.  Suppose that this is not the case.  Then, using \eqref{nub} we may deduce
$$ N(t) \sim_u 1$$
for all $t \geq 0$.  We conclude from Lemma \ref{con} that for every $m \geq 1$, there exists an $n_m$ such that
$$N_{v^{[t'_{n_m}]}}(t) \sim_u 1$$
for all $0 \leq t \leq m$.  But by \eqref{cdef} and \eqref{nvns-def} this implies that
$$ \osc( \eps m ) \lesssim_u 1$$
for all $m$ and some $\eps = \eps(u) > 0$ independent of $m$.  Note that $\eps$ is chosen as a lower bound
on the quantities $N(t_{n_m}'')^2/N(t_{n_m}')^2$ where $t''_{n_m}=t_{n_m}'+\frac{m}2N(t_{n_m}')^{-2}$.  This
contradicts the hypothesis $\lim_{T\to\infty} \osc(T) = \infty$ and so settles Case II.

\medskip

{\bf Case III:} $\lim_{T\to\infty} \osc(T) = \infty$ and $\inf_{t_0 \in J} a(t_0) > 0$.

In this case, there are no long periods of stability and no double cascades from high
to low frequencies; we will be able to extract a self-similar solution in the sense of Theorem~\ref{comp}.

Let $\eps = \eps(v) > 0$ be such that $\inf_{t_0 \in J} a(t_0) \geq 2\eps$.  We call a time $t_0$ \emph{future-focusing} if
\begin{equation}\label{future-def}
 N_v(t) \geq \eps N_v(t_0) \hbox{ for all } t \in J \hbox{ with } t \geq t_0
 \end{equation}
and \emph{past-focusing} if
\begin{equation}\label{past-def}
 N_v(t) \geq \eps N_v(t_0) \hbox{ for all } t \in J \hbox{ with } t \leq t_0.
\end{equation}
From the choice of $\eps$ we see that every time $t_0 \in J$ is either future-focusing or past-focusing,
or possibly both.

We will now show that either all sufficiently late times are future-focusing or that all sufficiently early times
are past-focusing.  If this were false, there would be a future-focusing time $t_0$ and a sequence of past-focusing
times $t_n$ that converges to the righthand endpoint of $J$.
For sufficiently large $n$, we have $t_n \geq t_0$.  By \eqref{future-def} and \eqref{past-def} we then see that
$$ N_v(t_n) \sim_v N_v(t_0)$$
for all such $n$.  For any $t_0 < t < t_n$, we know that $t$ is either past-focusing or
future-focusing; thus we have either $N_v(t_0) \geq \eps N_v(t)$ or $N_v(t_n) \geq \eps N_v(t)$.
Also, since $t_0$ is future-focusing, $N_v(t) \geq \eps N_v(t_0)$.  We conclude that
$$N_v(t) \sim_v N_v(t_0)$$
for all $t_0 < t < t_n$; since $t_n$ converges to $\sup(J)$, this
claim in fact holds for all $t_0 < t < \sup(J)$.  In particular,
from Corollary \ref{blow} we see that $v$ does not blow up forward
in finite time, that is, $\sup(J) = \infty$.  The function $N_v$ is now
bounded above and below on the interval $(t_0,+\infty)$.  This
implies that $\lim_{T \to \infty} \osc(T) < \infty$, a contradiction.
This proves the assertion at the beginning of the paragraph.

We may now assume that future-focusing occurs for all sufficiently late times; more
precisely, we can find a $t_0 \in J$ such that all times $t \geq t_0$ are future-focusing.
The case when all sufficiently early times are past-focusing reduces to this via time-reversal symmetry.

We will now recursively construct a new sequence of times $t_n$.  More precisely, we will explain how to choose
$t_{n+1}$ from $t_n$.

Since $\lim_{T\to\infty} \osc(T) = \infty$, we have $\osc(B) \geq 2/\eps$ for some sufficiently large $B = B(v) > 0$.
Given $J\ni t_n>t_0$ set $A=2B\eps^{-2}$ and $t_n'=t_n+\frac12AN_v(t_n)^{-2}$.  As $t_n'>t_0$, it is future-focusing
and so $N_v(t_n')\geq \eps N_v(t_n)$. From this, we see that
$$
\bigl\{t : |t-t_n'| \leq B N_v(t_n')^{-2} \bigr\} \subseteq \bigl[t_n,t_n+ A N_v(t_n)^{-2}\bigr]
$$
and thus, by the definition of $B$ and the fact that all $t\geq t_n$ are future-focusing,
\begin{equation}\label{Nwiggle}
 \sup_{t \in J\cap[t_n,t_n + A N_v(t_n)^{-2}]} N_v(t) \geq 2 N_v(t_n).
\end{equation}
Using this and Corollary~\ref{qc}, we see that for every $t_n \in J$ with $t_n \geq t_0$ there exists a time $t_{n+1} \in J$ obeying
\begin{equation}\label{nti}
t_n < t_{n+1} \leq t_n + A N(t_n)^{-2}
\end{equation}
such that
\begin{equation}\label{nti-2}
2 N_v(t_n) \leq N_v(t_{n+1}) \lesssim_v N_v(t_n)
\end{equation}
and
\begin{equation}\label{nti-3}
N_v(t) \sim_v N_v(t_n) \quad\text{for all $t_n \leq t \leq t_{n+1}$.}
\end{equation}

From \eqref{nti} we have
$$ N_v(t_n) \geq 2^n N_v(t_0)$$
for all $n \geq 0$, which by \eqref{nti-2} implies
$$ t_{n+1} \leq t_n + O_v( 2^{-2n} N_v(t_0)^{-2} ).$$
Thus $t_n$ converge to a limit and $N_v(t_n)$ to infinity.  In view of Lemma~\ref{qb},
this implies that $\sup(J)$ is finite and $\lim_{n \to \infty} t_n =\sup(J)$.

Let $n \geq 0$.  By \eqref{nti-2},
$$ N_v(t_{n+m}) \geq 2^m N_v(t_n)$$
for all $m \geq 0$ and so, using \eqref{nti} we obtain
$$
0 < t_{n+m+1} - t_{n+m} \lesssim_v  2^{-2m} N_v(t_n)^{-2}.
$$
Summing this series in $m$, we conclude that
$$ \sup(J) - t_n \lesssim_v N_v(t_n)^{-2}.$$
Combining this with Corollary~\ref{blow}, we obtain
$$ \sup(J) - t_n \sim_v N_v(t_n)^{-2}.$$
In particular, we have
$$ \sup(J)-t_{n+1} \sim_v \sup(J) - t_n \sim_v N_v(t_n)^{-2}.$$
Applying \eqref{nti-2} and \eqref{nti-3} shows
$$ \sup(J) - t \sim_v N_v(t)^{-2}$$
for all $t_n \leq t \leq t_{n+1}$.  Since $t_n$ converges to $\sup(J)$, we conclude that
$$
\sup(J)-t \sim_v N_v(t)^{-2}
$$
for all $t_0 \leq t < \sup(J)$.

As we have the freedom to modify $N(t)$ by a bounded function (modifying $C$ appropriately), we may normalise
$$ N_v(t) = (\sup(J)-t)^{-1/2}$$
for all $t_0 \leq t < \sup(J)$.  It is now not difficult to extract our sought-after self-similar solution
by suitably rescaling the interval $(t_0,\sup J)$ as follows.

Consider the normalisations $v^{[t_n]}$ of $v$ at times $t_n$ (cf. \eqref{untn}).
These are maximal-lifespan normalised solutions of mass $M(v)$, whose lifespans include the interval
$$ \Bigl(-\frac{\sup(J)-t_0}{\sup(J)-t_n},1\Bigr),$$
and which are almost periodic modulo scaling with compactness modulus function $C$ and frequency scale functions
\begin{equation}\label{nvsbound}
N_{v^{[t_n]}}(s) = (1-s)^{1/2}
\end{equation}
for all $-\frac{\sup(J)-t_0}{\sup(J)-t_n} < s < 1$. We now apply Lemma \ref{laps} and
conclude (passing to a subsequence if necessary) that $v^{[t_n]}$ converge locally
uniformly to a maximal-lifespan solution $u$ of mass $M(v)$ defined on an open interval
$I$ containing $(-\infty,1)$, which is almost periodic modulo symmetries.

By Lemma~\ref{con} and \eqref{nvsbound}, we see that $u$ has a frequency scale function $N$ obeying
$$ N(s) \sim_v (1-s)^{-1/2}$$
for all $s\in(-\infty,1)$. By modifying $N$ (and $C$) by a bounded factor, we may normalise
$$ N(s) = (1-s)^{-1/2}.$$
From this, Corollary~\ref{qc}, and Corollary~\ref{blow} we see that we must have $I =
(-\infty,1)$.  Applying a time translation (by $-1$) followed by a time reversal, we
obtain our sought-after self-similar solution.

\begin{remark}
Theorem~\ref{comp} may not be the final word on the matter.  It may be that by a further
passage to rescaled limits, one could improve the control in the case of a double high-to-low frequency cascade.
We will not pursue this matter here.
\end{remark}

\section{The self-similar solution}\label{ss-sec}

In this section we preclude self-similar solutions.  As mentioned in the Introduction, the key ingredient is
additional regularity.

\begin{theorem}[Regularity in the self-similar case]\label{ss-sob-thm}
Let $d=2$ and let $u$ be a spherically symmetric solution to \eqref{cubic nls} that is almost periodic modulo $G_\rad$ and
that is self-similar in the sense of Theorem~\ref{comp}.  Then $u(t) \in H^s_x(\R^2)$ for all $t \in (0,\infty)$ and all $s \geq 0$.
\end{theorem}

In order to preclude self-similar solutions we only need Theorem~\ref{ss-sob-thm} for $s=1$, as we now show.

\begin{corollary}[Absence of self-similar solutions]\label{NO-ss}
For $d=2$ there are no non-zero spherically symmetric solutions to \eqref{cubic nls} that are
self-similar in the sense of Theorem~\ref{comp}.
\end{corollary}

\begin{proof}
By Theorem~\ref{ss-sob-thm}, any such solution would obey $u(t) \in H^1_x(\R^2)$ for all $t\in (0,\infty)$.
Then, by the $H^1_x$ global well-posedness theory described after Conjecture~\ref{conj}, there exists a global
solution with initial data $u(t_0)$ at any time $t_0\in(0,\infty)$; recall that we assume $M(u) < M(Q)$ in the focusing case.
On the other hand, self-similar solutions blow up at time $t=0$.
These two facts (combined with the uniqueness statement in Theorem~\ref{local}) yield a contradiction.
\end{proof}

The remainder of this section is devoted to proving Theorem~\ref{ss-sob-thm}.

Let $u$ be as in Theorem~\ref{ss-sob-thm}.  For any $A > 0$, we define
\begin{equation}\label{cadef}
\begin{split}
\mathcal{M}(A) &:= \sup_{T > 0} \| u_{>A T^{-1/2}}(T) \|_{L^2_x(\R^2)} \\
\mathcal{S}(A) &:= \sup_{T > 0} \| u_{>A T^{-1/2}} \|_{L^4_{t,x}([T,2T] \times \R^2)}\\
\mathcal{N}(A) &:= \sup_{T > 0} \| P_{>A T^{-1/2}}(F(u)) \|_{L^{4/3}_{t,x}([T,2T] \times \R^2)}.
\end{split}
\end{equation}
The notation chosen indicates the quantity being measured, namely, the mass, the symmetric Strichartz norm,
and the nonlinearity in the adjoint Strichartz norm, respectively.  As $u$ is self-similar, $N(t)$ is comparable
to $T^{-1/2}$ for $t$ in the interval $[T,2T]$.  Thus, the Littlewood-Paley projections are adapted to the natural
frequency scale on each interval.

To prove Theorem~\ref{ss-sob-thm} it suffices to show that for every $s > 0$ we have
\begin{equation}\label{c2-targ}
\mathcal{M}(A) \lesssim_{s,u} A^{-s}
\end{equation}
whenever $A$ is sufficiently large depending on $u$ and $s$.  To establish this, we need a variety of estimates
linking $\mathcal{M}$, $\mathcal{S}$, and $\mathcal{N}$.  From mass conservation, Lemma~\ref{spacelemma}, self-similarity,
and H\"older's inequality, we see that
\begin{equation}\label{ctrivb}
\mathcal{M}(A) + \mathcal{S}(A) + \mathcal{N}(A)\lesssim_u 1
\end{equation}
for all $A > 0$.
From the Strichartz inequality (Lemma~\ref{L:strichartz}), we also see that
\begin{equation}\label{ci4}
\mathcal{S}(A) \lesssim \mathcal{M}(A) + \mathcal{N}(A)
\end{equation}
for all $A > 0$.  A similar application of Strichartz using \eqref{ctrivb} shows
\begin{equation}\label{ctrivb-2}
\|u \|_{L^2_t L^\infty_x([T, 2T] \times \R^2)} \lesssim_u 1
\end{equation}
for all $T > 0$.

Next, we use the bilinear Strichartz inequality to obtain a further connection between these quantities.

\begin{lemma}[Nonlinear estimate]\label{nle}  For all $A > 100$ we have
$$
\mathcal{N}(A) \lesssim_u
    \mathcal{S}(\tfrac A8)  \mathcal{M}(\sqrt{A}) + A^{-\frac1{4}} \bigl[\mathcal{M}(\tfrac A8) + \mathcal{N}(\tfrac A8)\bigr].
$$
\end{lemma}

\begin{proof}
Fix $A>100$.  It suffices to show that
\begin{equation}\label{patu4}
\| P_{>A T^{-\frac 12}}(F(u)) \|_{L^{4/3}_{t,x}([T,2T] \times \R^2)}
\lesssim_u \mathcal{S}(\tfrac A8) \mathcal{M}(\sqrt{A}) + A^{-\frac1{4}} \bigl[\mathcal{M}(\tfrac A8) + \mathcal{N}(\tfrac A8)\bigr]
\end{equation}
for arbitrary $T > 0$.  To do this, we decompose our solution as
$$u = u_{>\frac A8 T^{-1/2}} + u_{\sqrt{A} T^{-1/2} < \cdot \leq \frac A8 T^{-1/2}} + u_{\leq \sqrt{A} T^{-1/2}}.$$
Any term in the resulting expansion of $P_{>A T^{-1/2}}(F(u))$ that does not contain at least one factor of
$u_{>\frac A8 T^{-1/2}}$ vanishes.

Consider a term which contains at least one factor of $u_{\sqrt{A} T^{-1/2} < \cdot \leq \frac A8 T^{-1/2}}$.
By H\"older's inequality, the contribution of such a term to \eqref{patu4} can be bounded by
$$
\| u_{>\frac A8 T^{-1/2}} \|_{L^4_{t,x}([T,2T] \times \R^2)}
\| u_{> \sqrt{A} T^{-1/2}} \|_{L^\infty_t L^2_x([T,2T] \times \R^2)}
\| u \|_{L^2_t L^\infty_x([T,2T] \times \R^2)}.
$$
By \eqref{cadef}, the first factor is at most $\mathcal{S}(\frac A8)$ and the second is at most $\mathcal{M}(\sqrt{A})$,
while the last is $O_u(1)$ by \eqref{ctrivb-2}.

Finally, we consider the terms with one factor of $u_{>\frac A8 T^{-1/2}}$ and two factors of $u_{\leq \sqrt{A} T^{-1/2}}$.
By H\"older's inequality and \eqref{ctrivb}, the contribution of this term to \eqref{patu4} can be bounded by
$$ \| u_{> \frac A8T^{-1/2}} u_{\leq \sqrt{A} T^{-1/2}} \|_{L^2_{t,x}([T,2T] \times \R^2)}.$$
Applying the bilinear Strichartz inequality (Lemma~\ref{L:bilinear strichartz}), we bound this by
\begin{align*}
&A^{-1/4} \bigl(\| u_{>\frac A8 T^{-1/2}}(T)\|_{L^2_x(\R^2)}
    + \| P_{>\frac A8 T^{-1/2}} F(u) \|_{L^{4/3}_{t,x}([T,2T] \times \R^2)}\bigr)\\
&\quad \times \bigl(\| u_{\leq \sqrt{A} T^{-1/2}}(T)\|_{L^2_x(\R^2)}
    + \| u_{\leq \sqrt{A} T^{-1/2}} F(u) \|_{L^{4/3}_{t,x}([T,2T] \times \R^2)}\bigr),
\end{align*}
which is acceptable by virtue of \eqref{cadef} and \eqref{ctrivb}.
\end{proof}

We have some decay as $A \to \infty$:

\begin{lemma}[Qualitative decay]\label{qualit lemma}  We have
\begin{equation}\label{clima}
\lim_{A \to \infty} \mathcal{M}(A) = \lim_{A \to \infty} \mathcal{S}(A)=\lim_{A \to \infty} \mathcal{N}(A)=0.
\end{equation}
\end{lemma}

\begin{proof} The vanishing of the first limit follows from Definition~\ref{apdef}, self-similarity, and \eqref{cadef}.
The vanishing of the third limit then follows from that of the first and Lemma~\ref{nle}.  Lastly, we can deduce
$\lim_{A \to \infty} \mathcal{S}(A)=0$ using \eqref{ci4}.
\end{proof}

We have now gathered enough tools to prove some regularity, albeit in the symmetric Strichartz space.
As such, the next result is the crux of this section.

\begin{proposition}[Quantitative decay estimate]\label{quant prop}
Let $0 < \eta < 1$.  Then, if $A$ is sufficiently large depending on $u$ and $\eta$,
\begin{align}\label{recur}
\mathcal{S}(A) \leq \eta \mathcal{S}(\tfrac A{16}) + A^{-1/10}.
\end{align}
In particular, $\mathcal{S}(A) \lesssim_u  A^{-1/10}$ for all $A>0$.
\end{proposition}

\begin{proof} Fix $\eta\in (0,1)$.  It suffices to show
\begin{align}\label{recur-0}
\| u_{>A T^{-1/2}} \|_{L^4_{t,x}([T,2T] \times \R^2)} \lesssim_u \eta \mathcal{S}(\tfrac A{16}) + A^{-1/8}
\end{align}
for all $T > 0$, since then \eqref{recur} follows by redefining $\eta$ and requiring $A$ to be larger, both depending upon $u$.

Fix $T>0$.  By writing the Duhamel formula \eqref{old duhamel} beginning at $\frac T2$ and then using Lemma~\ref{L:strichartz},
we obtain
\begin{align*}
\| u_{>A T^{-1/2}} \|_{L^4_{t,x}([T,2T] \times \R^2)}
&\lesssim \| P_{>A T^{-1/2}} e^{i(t-\frac T2)\Delta} u(\tfrac T2) \|_{L^4_{t,x}([T,2T] \times \R^2)}\\
&\quad + \| P_{>A T^{-1/2}} F(u) \|_{L^{4/3}_{t,x}([\frac T2,2T] \times \R^2)}.
\end{align*}

First, we consider the second term.  By \eqref{cadef}, we have
$$ \| P_{>A T^{-1/2}} F(u) \|_{L^{4/3}_{t,x}([\frac T2,2T] \times \R^2)} \lesssim \mathcal{N}(A/2).$$
Using Lemma~\ref{nle} combined with Lemma~\ref{qualit lemma} (choosing $A$ sufficiently large depending on $u$ and $\eta$)
and \eqref{ctrivb}, we derive
$$
\| P_{>A T^{-1/2}} F(u) \|_{L^{4/3}_{t,x}([\frac T2,2T] \times \R^2)} \lesssim_u \text{RHS\eqref{recur-0}}.
$$
Thus, the second term is acceptable.

We now consider the first term.  It suffices to show
\begin{align}\label{ss decay}
\| P_{>A T^{-1/2}} e^{i(t-\frac T2)\Delta} u(\tfrac T2) \|_{L^4_{t,x}([T,2T] \times \R^2)} \lesssim_u A^{-1/8},
\end{align}
which we will deduce by first proving two estimates at a single frequency scale, interpolating between them, and then summing.

From Lemma \ref{L:Shuanglin} and mass conservation, we have
\begin{align}\label{ss Shuanglin}
\| P_{B T^{-1/2}} e^{i(t-\frac T2)\Delta} u(\tfrac T2) \|_{L^q_{t,x}([T,2T] \times \R^2)}
\lesssim_{u,q} (B T^{-1/2})^{1-\frac{4}{q}}
\end{align}
for all $10/3 < q < 4$ and $B>0$.  This is our first estimate.

Using the Duhamel formula \eqref{old duhamel}, we write
$$ P_{B T^{-1/2}} e^{i(t-\frac T2)\Delta} u(\tfrac T2) = P_{B T^{-1/2}} e^{i(t-\eps)\Delta} u(\eps)
- i \int_\eps^{\frac T2} P_{B T^{-1/2}} e^{i(t-t')\Delta} F(u(t'))\, dt'$$ for any $\eps > 0$.
By self-similarity, the former term converges strongly to zero in $L_x^2$ as $\eps \to 0$.
Convergence to zero in $L_x^\infty$ then follows from Lemma~\ref{Bernstein}. Thus, using the dispersive estimate
\eqref{dispersive}, we estimate
\begin{align*}
\| P_{B T^{-1/2}} e^{i(t-\frac T2)\Delta} & u(\tfrac T2)\|_{L^\infty_{t,x}([T,2T] \times \R^2)} \\
&\lesssim \Bigl\|\int_0^{\frac T2} \frac 1{t-t'} \|F(u(t'))\|_{L_x^1}\, dt'\Bigl\|_{L_t^\infty([T,2T])}\\
&\lesssim \frac{1}{T} \| F(u) \|_{L^1_{t,x}((0,\frac T2] \times \R^2)} \\
&\lesssim \frac{1}{T} \sum_{0<\tau\leq\frac T4} \| F(u) \|_{L^1_{t,x}([\tau,2\tau] \times \R^2)} \\
&\lesssim \frac{1}{T} \sum_{0<\tau\leq\frac T4} \tau^{1/2} \| u \|^2_{L^4_{t,x}([\tau,2\tau] \times \R^2)}
        \| u \|_{L^\infty_tL^2_x([\tau,2\tau] \times \R^2)} \\
&\lesssim_u T^{-1/2}.
\end{align*}
The last step used that $\| u \|_{4,4}\sim 1$ on dyadic intervals, as follows from self-similarity plus Lemma~\ref{spacelemma}.

Interpolating between the estimate just proved and \eqref{ss Shuanglin} with $q=\frac72$, we obtain
$$
\| P_{B T^{-1/2}} e^{i(t-\frac T2)\Delta} u(\tfrac T2) \|_{L^4_{t,x}([T,2T] \times \R^2)}
\lesssim_u B^{-1/8}.
$$
Summing this over dyadic $B\geq A$ yields \eqref{ss decay} and hence \eqref{recur-0}.

We now explain why \eqref{recur} implies $\mathcal{S}(A) \lesssim_u  A^{-1/10}$. Choose $\eta=\frac12$.
Then there exists $A_0$, depending on $u$, so that \eqref{recur} holds for $A\geq A_0$. By \eqref{ctrivb},
we need only bound $\mathcal{S}(A)$ for $A\geq A_0$.

Choose $k\geq 1$ so that $2^{-4k} A \leq A_0 \leq 2^{-4(k-1)}A$.  By iterating \eqref{recur} $k$ times and
then using \eqref{ctrivb},
\begin{align*}
\mathcal{S}(A) &\leq \bigl[1+\eta2^{4/10}+\cdots + (\eta2^{4/10})^{k-1} \bigr] A^{-1/10} + \eta^k \mathcal{S}(2^{-4k}A) \\
&\lesssim_u A^{-1/10} + \eta^{k} \lesssim_u A^{-1/10}.
\end{align*}
Note that the last inequality uses the way we chose $\eta$ and $k$.
\end{proof}

\begin{corollary}\label{decay M,S,N}
For any $A>0$ we have
\begin{align*}
\mathcal{M}(A)+\mathcal{S}(A)+\mathcal{N}(A)\lesssim_u A^{-1/10}.
\end{align*}
\end{corollary}

\begin{proof}
The bound on $\mathcal{S}$ was derived in Proposition~\ref{quant prop}.  The bound on $\mathcal{N}$
follows from this, Lemma~\ref{nle}, and \eqref{ctrivb}.

We now turn to the bound on $\mathcal{M}$.  By Lemma~\ref{duhamel L},
\begin{align}\label{ss forw}
\|P_{>AT^{-1/2}}u(T)\|_2
\lesssim \sum_{k=0}^\infty \Bigl\| \int_{2^kT}^{2^{k+1}T} e^{i(T-t')\Delta} P_{>AT^{-1/2}} F(u(t'))\,dt' \Bigr\|_2,
\end{align}
where weak convergence has become strong convergence because of the frequency projection and the fact that
$N(t)=t^{-1/2}\to 0$ as $t\to \infty$.  Intuitively, the reason for using \eqref{duhamel} forward in time
is that the solution becomes smoother as $N(t)\to 0$.

Combining \eqref{ss forw} with Lemma~\ref{L:strichartz} and \eqref{cadef}, we get
\begin{align}\label{N-->M}
\mathcal{M}(A)=\sup_{T>0}\|P_{>AT^{-1/2}}u(T)\|_2
\lesssim \sum_{k=0}^\infty \mathcal{N}(2^{k/2}A).
\end{align}
The desired bound on $\mathcal{M}$ now follows from that on $\mathcal{N}$.
\end{proof}

We have succeeded in proving that self-similar solutions have almost $\frac1{10}$ derivatives in $L^2_x$.
Upgrading this to arbitrarily many derivatives now follows by a simple inductive argument.

\begin{proof}[Proof of Theorem~\ref{ss-sob-thm}]
Combining Lemma~\ref{nle} with Corollary~\ref{decay M,S,N} shows
$$
\mathcal{N}(A) \lesssim_u  A^{-1/20}\bigl[\mathcal{S}(\tfrac A8) + \mathcal{M}(\tfrac A8) + \mathcal{N}(\tfrac A8)\bigr].
$$
Together with \eqref{ci4} and \eqref{N-->M}, this allows us to deduce
$$
\mathcal{S}(A)+\mathcal{M}(A)+\mathcal{N}(A) \lesssim_u A^{-\sigma}
\quad\Longrightarrow\quad
\mathcal{S}(A)+\mathcal{M}(A)+\mathcal{N}(A) \lesssim_u A^{-\sigma-\frac1{20}}
$$
for any $\sigma>0$.  Iterating this statement shows that $u(t)\in H^{s}_x(\R^2)$ for all $s>0$.
\end{proof}

\section{An in/out decomposition}\label{pseudo-sec}

Spherical symmetry forces the amplitude of a wave to diminish as it moves away from the origin.  This section
provides a tool that we will use to exploit this idea in the next section.  To set the stage
however, let us briefly try to convey the guiding principle.  Given $t$,
we will split $u(t)$ into incoming and outgoing waves.  Writing the Duhamel formula \eqref{duhamel} with integration over the past,
the source terms producing these incoming waves will lie at large radii.  By contrast, it is favourable
to represent outgoing waves using Duhamel with integration over the future.

For spherically symmetric functions, we have the Fourier-Bessel formula
$$
    \hat f(\xi) = (2\pi)^{-1}\int_{\R^2} J_0( |x| |\xi|) f(x)\, dx,
$$
where $J_0$ is the Bessel function of the first kind
$$ J_0(r) := \int_0^{2\pi} e^{ ir \cos(\theta)} \, \frac{d\theta}{2\pi}$$
(see e.g. \cite[Ch. IV]{stein:weiss}).
Similarly we have the Fourier-Bessel inversion formula
$$
f(x) = (2\pi)^{-1} \int_{\R^2} J_0( |x| |\xi|) \hat f(\xi)\, d\xi,
$$
which expresses $f$ as a linear combination of standing waves.  These standing waves can then be decomposed
into their incoming and outgoing components:
\begin{equation}\label{Bessel from Hankel}
J_0( |\xi||x|) = \tfrac12 H^{(1)}_0(|\xi||x|) + \tfrac12 H^{(2)}_0(|\xi||x|),
\end{equation}
Here, $H^{(1)}_0$ denotes the Hankel function of the first kind and order zero; for real arguments,
the Hankel function of the second kind is its complex conjugate.

With these considerations in mind, it is natural to make the following

\begin{definition}  Let $P^+$ denote the projection onto outgoing spherical waves,
\begin{equation}\label{P+ defn}
\begin{aligned}{}
[P^+ f](x) &= \tfrac12 (2\pi)^{-1} \int_{\R^2} H^{(1)}_0(|\xi||x|) \hat f(\xi)\,d\xi \\
&= \tfrac12 (2\pi)^{-2}\int_{\R^2\times\R^2} H^{(1)}_0( |\xi||x|) J_0\bigl( |\xi| |y|\bigr)f(y) \,d\xi \,dy,
\end{aligned}
\end{equation}
and let $P^-$ denote the projection onto incoming spherical waves, whose kernel is the complex conjugate of that of $P^+$.
We will write $P_N^\pm$ for the product $P^\pm P_N$, where $P_N$ is a Littlewood--Paley projection.
\end{definition}

Although we refer to $P^\pm$ as projections, they are not idempotent, nor are they self-adjoint as operators on $L^2_x$.
In these two regards, our $P^\pm$ differ from the incoming/outgoing projections discussed in \cite{Mourre,Perry}.  Nevertheless,
the naive approach taken here suffices and lends itself well to the derivation of kernel estimates, to which we now turn.

\begin{prop}[Properties of $P^\pm$]\label{P:P properties}
\
\begin{SL}
\item The $P^{\pm}$ are bounded on $L^2(\R^2)$.
\item $P^+ + P^- $ represents the projection from $L^2$ onto $L^2_\rad$.
\item For $|x|\gtrsim N^{-1}$ and $t\gtrsim N^{-2}$, the integral kernel obeys
\begin{equation*}
\bigl| [P^\pm_N e^{\mp it\Delta}](x,y) \bigr| \lesssim \begin{cases}
    \bigl( |x||y||t| \bigr)^{-\frac12}  &: \  |y|-|x|\sim  Nt \\[1ex]
    \frac{N^2}{(N|x|)^{1/2}\langle N|y|\rangle^{1/2}} \bigl\langle N^2t + N|x| - N|y| \bigr\rangle^{-m}
            &: \  \text{otherwise}\end{cases}
\end{equation*}
for all $m\geq 0$.
\item For $|x|\gtrsim N^{-1}$ and $|t|\lesssim N^{-2}$, the integral kernel obeys
\begin{equation*}
\bigl| [P^\pm_N e^{\mp it\Delta}](x,y) \bigr|
    \lesssim  \frac{N^2}{(N|x|)^{1/2}\langle N|y|\rangle^{1/2}} \bigl\langle N|x| - N|y| \bigr\rangle^{-m}
\end{equation*}
for all $m\geq 0$.
\end{SL}
\end{prop}

\begin{proof}
By performing the $\xi$ integral in \eqref{P+ defn}, we obtain
\begin{equation}\label{P+ defn 2}
[P^\pm f](x) = \frac{f(x)}2 \pm \frac{i}{2\pi^2} \!\int_{\R^2} \frac{f(y)\,dy}{|x|^2-|y|^2}
\end{equation}
for radial Schwartz functions $f$,  while $P^\pm f \equiv 0$ if $f$ is orthogonal to radial functions, that is,
has zero mean on all circles centered at the origin.   Switching to polar coordinates and changing variables reveals that the
integral in \eqref{P+ defn 2} is simply the Hilbert transform in disguise.  This shows that $P^+$ is bounded on $L^p(\R^2)$
for any $1<p<\infty$, settling claim (i).  Observation (ii) also follows from this discussion.

The third claim is an exercise in stationary phase.  We will only provide the details for $P_N^+ e^{-it\Delta}$;
the other kernel is its complex conjugate. By \eqref{P+ defn} we have the following formula for the kernel:
\begin{align}\label{kernel}
[P^+_N e^{-it\Delta}](x,y) = \tfrac12 (2\pi)^{-2}\int_{\R^2} H^{(1)}_0( |\xi| |x|) e^{ it|\xi|^2} J_0\bigl(|\xi||y|\bigr)
    \psi\bigl(\tfrac{\xi}N\bigr)\,d\xi,
\end{align}
where $\psi$ is the multiplier from the Littlewood--Paley projection.  To proceed, we use the following information
about Bessel/Hankel functions.
\begin{align}\label{Bessel symbol}
J_0(r) = \frac{a(r) e^{ir}}{\langle r\rangle^{1/2}} + \frac{\bar a(r) e^{-ir}}{\langle r\rangle^{1/2}},
\end{align}
where $a(r)$ obeys the symbol estimates
\begin{equation}\label{symbol type}
\Bigr| \frac{\partial^m a(r)}{\partial r^m} \Bigr| \lesssim \langle r \rangle^{-m}
    \quad \text{for all $m\geq0$.}
\end{equation}
The Hankel function $H^{(1)}_0(r)$ has a logarithmic singularity at $r=0$; however, for $r\gtrsim 1$,
\begin{align}\label{Hankel symbol}
H^{(1)}_0(r) = \frac{b(r) e^{ir}}{r^{1/2}}
\end{align}
for a smooth function $b(r)$ obeying \eqref{symbol type}.  As we assume $|x|\gtrsim N^{-1}$, the logarithmic singularity
does not enter into our considerations.

Substituting \eqref{Bessel symbol} and \eqref{Hankel symbol} into \eqref{kernel}, we see that a stationary phase point
can only occur in the term containing $\bar a(r)$ and even then only if $|y|-|x|\sim Nt$.  In this case,
stationary phase yields the first estimate.  In all other cases, integration by parts yields the second estimate.

Part (iv) is also a consequence of \eqref{kernel} and stationary phase techniques.  Since $t$ is so small,
$e^{i|\xi|^2t}$ shows no appreciable oscillation and can be incorporated into $\psi(\frac \xi N)$.
For $\bigl||y|-|x|\bigr|\leq N^{-1}$, the result follows from the naive $L^1$ estimate.
For larger $|x|-|y|$, one integrates by parts $m$ times.
\end{proof}

%

\section{Additional regularity}\label{glob-sob-sec}

This section is devoted to a proof of

\begin{theorem}[Regularity in the global case]\label{glob-sob-thm}
Let $d=2$ and let $u$ be a global spherically symmetric solution to \eqref{cubic nls} that is almost periodic modulo $G_\rad$.
Suppose also that $N(t)\lesssim 1$ for all $t\in\R$.  Then $u \in L^\infty_t H^s_x(\R \times \R^2)$ for all $s \geq 0$.
\end{theorem}

That $u(t)$ is smooth will follow from a careful study of the Duhamel formulae \eqref{duhamel}.  Different
phenomena are involved at times near $t$ and those far from $t$.  Near $t$, the important fact is that
there is little mass at high frequencies, as is implied by the definition of almost periodicity and the
boundedness of the frequency scale function $N(t)$.
Far from $t$, the driving force behind regularity is the spherical symmetry of the solution.  Of course
this symmetry is only valuable at large radii.  Consequently, we are only able to exploit it by using
the in/out decomposition described in Section~\ref{pseudo-sec}.

Let us now begin the proof.  For the remainder of the section, $u$ will denote a solution to \eqref{cubic nls}
that obeys the hypotheses of Theorem~\ref{glob-sob-thm}.

We first record some basic local estimates.  From mass conservation we have
\begin{equation}\label{masst}
 \| u \|_{L^\infty_t L^2_x(\R \times \R^2)} \lesssim_u 1,
 \end{equation}
while from Definition~\ref{apdef} and the fact that $N(t)$ is bounded we have
$$ \lim_{N \to \infty} \| u_{\geq N} \|_{L^\infty_t L^2_x(\R \times \R^2)}  = 0.$$
From Lemma~\ref{spacelemma} and $N(t)\lesssim 1$, we have
\begin{align}\label{gr 44}
\| u \|_{L^4_{t,x}(J \times \R^2)} \lesssim_u \langle |J|\rangle^{1/4}
\end{align}
for all intervals $J\subset \R$.  By H\"older, this implies
$$ \| F(u) \|_{L^{4/3}_{t,x}(J \times \R^2)} \lesssim_u \langle |J|\rangle^{3/4}$$
and then, by the (endpoint) Strichartz inequality (Lemma \ref{L:strichartz}),
\begin{equation}\label{2infty}
\| u \|_{L^2_t L^\infty_x(J \times \R^2)} \lesssim_u \langle |J|\rangle^{1/2}.
\end{equation}
More precisely, one first treats the case $|J| = O(1)$ using \eqref{gr 44} and then larger intervals by subdivision.
Similarly, from the weighted Strichartz inequality (Lemma~\ref{L:wes}),
\begin{equation}\label{wait}
\| |x|^{1/2} u \|_{L^4_t L^\infty_x(J \times \R^2)} \lesssim_u \langle |J|\rangle^{1/4}.
\end{equation}

Now, for any dyadic number $N$, define
\begin{equation}\label{cndef}
\mathcal{M}(N) :=\| u_{\geq N} \|_{L^\infty_t L^2_x(\R \times \R^2)}.
\end{equation}
From the discussion above, we see that $\mathcal{M}(N) \lesssim_u 1$ and
\begin{equation}\label{clim}
\lim_{N \to \infty} \mathcal{M}(N) = 0.
\end{equation}

To prove Theorem~\ref{glob-sob-thm}, it suffices to show that $\mathcal{M}(N) \lesssim_{u,s} N^{-s}$
for any $s > 0$ and all $N$ sufficiently large depending on $u$ and $s$.  This will immediately follow
from iterating the following proposition with a suitably small choice of $\eta$ (depending on $u$ and $s$):

\begin{proposition}[Regularity]\label{ueta}  Let $u$ be as in Theorem~\ref{glob-sob-thm} and let $\eta>0$ be a small number.
Then
$$ \mathcal{M}(N) \lesssim_{u} \eta \mathcal{M}(\tfrac N8)$$
whenever $N$ is sufficiently large depending on $u$ and $\eta$.
\end{proposition}

The rest of this section is devoted to proving Proposition~\ref{ueta}.  Our task is to show that
$$
    \| u_{\geq N}(t_0) \|_{L^2_x(\R^2)} \lesssim_{u} \eta\mathcal{M}(\tfrac N8)
$$
for all times $t_0$ and all $N$ sufficiently large (depending on $u$ and $\eta$).
By time translation symmetry, we may assume $t_0=0$.  As suggested above, one of the
keys to obtaining additional regularity is Lemma~\ref{duhamel L}.  Specifically, we have
\begin{align}
u_{\geq N}(0) &=\bigl( P^+ + P^- \bigr) u_{\geq N}(0) \label{pm rep}\\
       &= \lim_{T\to\infty} i\int_0^T P^+ e^{-it\Delta} P_{\geq N}F(u(t))\,dt
                 -i \lim_{T\to\infty} \int_{-T}^0 P^- e^{-it\Delta} P_{\geq N}F(u(t))\,dt,\notag
\end{align}
where the limit is to be interpreted as a weak limit in $L^2$.
However, this representation is not useful for $|x|$ small because the kernels of $P^\pm$ have a logarithmic singularity at $x=0$.
To deal with this, we will use a different representation for $|x|\leq N^{-1}$, namely
\begin{align}\label{no pm rep}
u_{\geq N}(0) &= \lim_{T\to\infty} i\int_0^T e^{-it\Delta} P_{\geq N}F(u(t))\,dt,
\end{align}
also as a weak limit.  To deal with the poor nature of these limits, we note that
\begin{equation}\label{weakly closed}
f_T \to f \text{ weakly} \quad\Longrightarrow\quad
    \|f\| \leq \limsup_{T\to\infty} \|f_T\|,
\end{equation}
or equivalently, that the unit ball is weakly closed.

Despite the fact that different representations will be used depending on the size of $|x|$, some
estimates can be dealt with in a uniform manner.  The first such example is a bound on
integrals over short times.

\begin{lemma}[Local estimate]\label{local-lemma}  For any $\eta>0$, there exists $\delta = \delta(u,\eta) > 0$ such that
\begin{equation*}
  \Bigl\| \int_0^\delta e^{-it\Delta} P_{\geq N} F(u(t))\,dt \Bigr\|_{L^2_x} \lesssim_u \eta \mathcal{M}(\tfrac N8),
\end{equation*}
provided $N$ is sufficiently large depending on $u$ and~$\eta$.  An analogous estimate holds for integration over
$[-\delta,0]$ and after pre-multiplication by $P^\pm$ (they are bounded operators on $L^2_x$).
\end{lemma}

\begin{proof}
By Lemma~\ref{L:strichartz}, it suffices to prove
$$
 \| P_{\geq N} F(u) \|_{L^2_t L^1_x(J \times \R^2)} \lesssim_u \eta \mathcal{M}(\tfrac N8)
$$
for any interval $J$ of length $|J| \leq \delta$ and all sufficiently large $N$ depending on $u$ and~$\eta$.

From \eqref{clim}, there exists $N_0 = N_0(u,\eta)$ such that
\begin{equation}\label{uetan}
 \| u_{\geq N_0} \|_{L^\infty_t L^2_x(\R \times \R^2)} \leq \eta.
\end{equation}
Let $N>\tfrac 18 N_0$.  We decompose
$$ u = u_{\geq \frac N8} + u_{N_0 \leq \cdot < \frac N8} + u_{<N_0}.$$
Any term in the resulting decomposition of $P_{\geq N} F(u)$ which does not involve at least one $u_{\geq \frac N8}$ will vanish.

Consider first a term with two factors of the form $u_{<N_0}$.  Using H\"older's inequality, \eqref{masst}, \eqref{cndef},
and Lemma~\ref{Bernstein}, we estimate
\begin{align*}
\|u_{\geq \frac N8} & u_{<N_0}^2 \|_{L^2_t L^1_x(J \times \R^2)}\\
&\lesssim |J|^{1/2} \| u_{\geq \frac N8} \|_{L^\infty_t L^2_x(\R \times \R^2)} \| u_{<N_0} \|_{L^\infty_t L^2_x(\R \times \R^2)}
    \| u_{<N_0} \|_{L^\infty_t L^\infty_x(\R \times \R^2)} \\
&\lesssim_u |J|^{1/2} \mathcal{M}(\tfrac N8) N_0.
\end{align*}
Choosing $\delta$ sufficiently small  depending on $\eta$ and $N_0$, we see that this term is acceptable.

It remains only to consider those components of $P_{\geq N} F(u)$ which involve $u_{\geq \frac N8}$
and at least one other term which is not $u_{<N_0}$.  We estimate such terms using H\"older's inequality,
\eqref{2infty}, \eqref{cndef}, and \eqref{uetan} in the following fashion:
\begin{align*}
\| u_{\geq \frac N8} u_{\geq N_0} u\|_{L^2_t L^1_x(J \times \R^2)}
&\lesssim \| u_{\geq N_0}\|_{L^\infty_t L^2_x(\R\times \R^2)} \| u_{\geq \frac N8} \|_{L^\infty_t L^2_x(\R \times \R^2)}
    \|  u\|_{L^2_t L^\infty_x(J \times \R^2)}\\
&\lesssim_u  \eta \mathcal{M}(\tfrac N8) \langle |J|\rangle^{1/2}.
\end{align*}
This completes the proof of the lemma.
\end{proof}

We now turn our attention to $|t|\geq \delta$.  In this case we make the decomposition
$$
P_{\geq N} = \sum_{M\geq N} P_M \tilde P_M,
$$
where $\tilde P_M:=P_{M/2}+P_M+P_{2M}$.  In this way, \eqref{no pm rep} becomes
\begin{equation}\label{no pm rep'}
\begin{aligned}
u_{\geq N}(0,x) &= i\int_0^\delta e^{-it\Delta} P_{\geq N}F(u(t))\,dt \\
    &\quad + \lim_{T\to\infty} \sum_{M\geq N} i\int_\delta^T\int_{\R^2} [P_M e^{-it\Delta}](x,y) [\tilde P_M F(u(t))](y) \,dy\,dt,
\end{aligned}
\end{equation}
which we will use when $|x|\leq N^{-1}$. The analogous reformulation of \eqref{pm rep}, namely
\begin{equation}\label{pm rep'}
\begin{aligned}
u_{\geq N}(0,x) &= i\int_0^\delta P^+ e^{-it\Delta} P_{\geq N}F(u(t))\,dt
        -i\int_{-\delta}^0 P^- e^{-it\Delta} P_{\geq N}F(u(t))\,dt \\
&\quad+\lim_{T\to\infty}\sum_{M\geq N} i\int_\delta^T      \int_{\R^2} [P_M^+ e^{-it\Delta}](x,y) [\tilde P_M F(u(t))](y) \,dy\,dt \\
&\quad-\lim_{T\to\infty}\sum_{M\geq N} i\int_{-T}^{-\delta}\int_{\R^2} [P_M^- e^{-it\Delta}](x,y) [\tilde P_M F(u(t))](y) \,dy\,dt,
\end{aligned}
\end{equation}
will be used when $|x|$ is large.

To estimate the integrals where $|t|\geq\delta$, we break the region of $(t,y)$ integration into two pieces,
namely, where $|y|\gtrsim M|t|$ and $|y|\ll M|t|$.  The former is the more significant region; it contains the points
where the integral kernels $P_Me^{-it\Delta}(x,y)$ and $P_M^\pm e^{-it\Delta}(x,y)$ are large (see Lemmas~\ref{lr:propag est L}
and~\ref{P:P properties}).  More precisely, when $|x|\leq N^{-1}$, we use \eqref{no pm rep'}; in this case
$|y-x|\sim M|t|$ implies $|y|\gtrsim M|t|$ for $|t|\geq \delta \geq N^{-2}$.  (This last condition can be subsumed
under our hypothesis $N$ sufficiently large depending on $u$ and $\eta$.)  When $|x|\geq N^{-1}$, we use \eqref{pm rep'};
in this case $|y|-|x|\sim M|t|$ implies $|y|\gtrsim M|t|$.

The next lemma bounds the integrals over the significant region $|y|\gtrsim M|t|$.  Let $\chi_k$ denote
the characteristic function of the set
$$\{(t,y):\, 2^k\delta\leq |t|\leq 2^{k+1}\delta, \ |y|\gtrsim M|t|\}.$$

\begin{lemma}[Main contribution]\label{ar main}
Let $\eta>0$ be a small number and let $\delta$ be as in Lemma~\ref{local-lemma}.  Then
\begin{equation*}
\sum_{M\geq N}\sum_{k=0}^\infty \Bigl\| \int_{2^k\delta}^{2^{k+1}\delta}\!\!\int_{\R^2}
    [P_M e^{-it\Delta}](x,y)\,\chi_k(t,y) \, [\tilde P_M F(u(t))](y) \,dy\,dt\Bigr\|_{L^2_x}
\lesssim_u \eta \mathcal{M}(\tfrac N8)
\end{equation*}
for all $N$ sufficiently large depending on $u$ and $\eta$.  An analogous estimate holds with $P_M$
replaced by $P_M^+$ or $P_M^-$; moreover, the time integrals may be taken over $[-2^{k+1}\delta,-2^k\delta]$.
\end{lemma}

\begin{proof}
By the adjoint of the weighted Strichartz estimate \eqref{wait S} and then H\"older's inequality and \eqref{cndef},
\begin{align*}
\Bigl\| \int_{2^k\delta}^{2^{k+1}\delta}&\!\!\int_{\R^2}[P_M e^{-it\Delta}](x,y)\,\chi_k(t,y)
    \, [\tilde P_M F(u(t))](y) \,dy\,dt\Bigr\|_{L^2_x}\\
&\lesssim (M 2^k\delta)^{-\frac 12} \|\chi_k \tilde P_M F(u)\|_{L_t^{4/3}L_x^1}\\
&\lesssim (M 2^k\delta)^{-\frac 12} \|u_{\geq \frac M8}\|_{L_t^\infty L_x^2} \|u\|_{L_t^\infty L_x^2}
     \bigl\{\|\chi_k u_{\geq M}\|_{L_t^{4/3}L_x^\infty}+\|\chi_k u_{< M}\|_{L_t^{4/3}L_x^\infty}\bigr\}\\
&\lesssim_u (M 2^k\delta)^{-\frac 12} \mathcal{M}( \tfrac N8)
        \bigl\{M^{-\frac 12}\||y|^{\frac 12} u_{\geq M}\|_{L_t^4L_x^\infty([2^k\delta, 2^{k+1}\delta]\times\R^2)}\\
&\qquad \qquad \qquad \qquad \qquad \qquad \qquad
    + M^{-\frac 12} ( 2^k\delta)^{\frac 14} \||y|^{\frac 12} u_{< M}\|_{L_{t,x}^\infty}\bigr\}.
\end{align*}
By \eqref{wait},
$$
\||y|^{\frac 12} u_{\geq M}\|_{L_t^4L_x^\infty([2^k\delta, 2^{k+1}\delta]\times\R^2)}
        \lesssim_u \langle 2^k\delta \rangle^{\frac 14},
$$
while by Lemma~\ref{radial embedding} (with $d=2$) and the conservation of mass,
$$
\bigl\||y|^{\frac 12} u_{< M}\bigr\|_{L_{t,x}^\infty}\lesssim_u M^{\frac 12}.
$$
Thus,
\begin{align*}
\Bigl\| \int_{2^k\delta}^{2^{k+1}\delta}\!\!\int_{\R^2}[P_M e^{-it\Delta}](x,y)\,\chi_k(t,y)
    &\, [\tilde P_M F(u(t))](y) \,dy\,dt\Bigr\|_{L^2_x}\\
&\lesssim_u (M 2^k\delta)^{-\frac 12} \mathcal{M}( \tfrac N8)\bigl(M^{-\frac 12}\langle 2^k\delta \rangle^{\frac 14}
        +( 2^k\delta)^{\frac 14}\bigr).
\end{align*}
Summing first over $k\geq 0$ and then over $M\geq N$, we obtain
\begin{align*}
\sum_{M\geq N}\sum_{k=0}^\infty \Bigl\| \int_{2^k\delta}^{2^{k+1}\delta}\!\!\int_{\R^2}
    [P_M e^{-it\Delta}]&(x,y)\, \chi_k(t,y) \, [\tilde P_M F(u(t))](y) \,dy\,dt\Bigr\|_{L^2_x}\\
&\lesssim_u \bigl(N^{-1}\delta^{-\frac 12} + N^{-\frac 12}\delta^{-\frac 14}\bigr)\mathcal{M}( \tfrac N8).
\end{align*}
Choosing $N$ sufficiently large depending on $\delta$ and $\eta$ (and hence only on $u$ and $\eta$), we obtain the desired bound.

The last claim follows from the $L_x^2$-boundedness of $P^\pm$ and the time reversal symmetry of the argument just presented.
\end{proof}

We turn now to the region of $(t,y)$ integration where $|y|\ll M|t|$.  First, we describe the bounds that we will use
for the kernels of the propagators.  For $|x|\leq N^{-1}$, $|y|\ll M|t|$, and $|t|\geq \delta \gg N^{-2}$,
\begin{equation}\label{PM propig}
|P_M e^{-it\Delta}(x,y)| \lesssim \frac{1}{(M^2|t|)^{50}} \frac{M^2}{\langle M(x-y)\rangle^{50}};
\end{equation}
this follows from Lemma~\ref{lr:propag est L} since under these constraints, $|y-x|\ll M|t|$.  For $|x|\geq N^{-1}$
and $y$ and $t$ as above,
\begin{equation}\label{PMpm propig}
|P_M^\pm e^{-it\Delta}(x,y)| \lesssim \frac{1}{(M^2|t|)^{50}}
        \frac{M^2}{\langle Mx\rangle^{1/2}\langle My\rangle^{1/2}\langle M|x|-M|y|\rangle^{50}};
\end{equation}
by Proposition~\ref{P:P properties}.  Note that we have used $|y|-|x|\ll M|t|$ and
$$
\langle M^2|t|+M|x|-M|y| \rangle^{-100} \lesssim (M^2|t|)^{-50}  \langle M|x|-M|y| \rangle^{-50}
$$
in order to simplify the bound.

From \eqref{PM propig} and \eqref{PMpm propig} we see that under the hypotheses set out above,
\begin{equation}\label{combined propig}
|P_M e^{-it\Delta}(x,y)| + |P_M^\pm e^{-it\Delta}(x,y)| \lesssim \frac{1}{(M^2|t|)^{50}} K_M(x,y),
\end{equation}
where
$$
K_M(x,y) := \frac{M^2}{\langle M(x-y)\rangle^{50}}
        + \frac{M^2}{\langle Mx\rangle^{1/2}\langle My\rangle^{1/2}\langle M|x|-M|y|\rangle^{50}}.
$$
Note that by Schur's test, this is the kernel of a bounded operator on $L^2_x(\R^2)$.

Let $\tilde\chi_k$ denote the characteristic function of the set
$$\{(t,y):\, 2^k\delta\leq |t|\leq 2^{k+1}\delta, \ |y|\ll M|t|\}.$$

\begin{lemma}[The tail]\label{ar tail}
Let $\eta>0$ be a small number and let $\delta$ be as in Lemma~\ref{local-lemma}.  Then
\begin{equation*}
\sum_{M\geq N}\sum_{k=0}^\infty \Bigl\| \int_{\R}\int_{\R^2}
    \frac{K_M(x,y)}{(M^2|t|)^{50}}\,\tilde\chi_k(t,y) \, [\tilde P_M F(u(t))](y) \,dy\,dt\Bigr\|_{L^2_x}
\lesssim_u \eta \mathcal{M}(\tfrac N8)
\end{equation*}
for all $N$ sufficiently large depending on $u$ and $\eta$.
\end{lemma}

\begin{proof}
Using H\"older's inequality, \eqref{2infty}, and \eqref{cndef},
\begin{align*}
\Bigl\| \int_{\R}\int_{\R^2}
    & \frac{K_M(x,y)}{(M^2|t|)^{50}}\,\tilde\chi_k(t,y) \, [\tilde P_M F(u(t))](y) \,dy\,dt\Bigr\|_{L^2_x}\\
&\lesssim (M^2 2^k\delta )^{-50} \|\tilde\chi_k \tilde P_M F(u)\|_{L_t^1L^2_x}\\
&\lesssim (M^2 2^k\delta )^{-50} \|u_{\geq \frac M8}\|_{L_t^\infty L_x^2}
            \|u\|^2_{L_t^2L_x^\infty([2^k\delta, 2^{k+1}\delta]\times\R^2)}\\
&\lesssim_u (M^2 2^k\delta )^{-50} \mathcal{M}(\tfrac N8)\langle 2^k\delta\rangle.
\end{align*}
Summing over $k\geq 0$ and $M\geq N$, we get
\begin{align*}
\sum_{M\geq N}\sum_{k=0}^\infty \Bigl\| \int_{\R}\int_{\R^2}
    \frac{K_M(x,y)}{(M^2|t|)^{50}}\,\tilde\chi_k(t,y) \, [\tilde P_M F(u(t))](y) \,dy\,dt\Bigr\|_{L^2_x}
\lesssim_u (N^2\delta)^{-50}\mathcal{M}(\tfrac N8).
\end{align*}
The claim follows by choosing $N$ sufficiently large depending on $\delta$ and $\eta$ (and hence only on $u$ and $\eta$).
\end{proof}

We have now gathered enough information to complete the

\begin{proof}[Proof of Proposition~\ref{ueta}]
Naturally, we may bound $\|u_{\geq N}\|_{L^2}$ by separately bounding the $L^2$ norm on the ball $\{|x|\leq N^{-1}\}$
and on its complement.   On the ball, we use \eqref{no pm rep'}, while outside the ball we use \eqref{pm rep'}.
Invoking \eqref{weakly closed} and the triangle inequality, we reduce the proof to bounding certain integrals.
The integrals over short times were estimated in Lemma~\ref{local-lemma}.  For $|t|\geq \delta$,
we further partition the region of integration into two main pieces.  The first piece, where $|y|\gtrsim M|t|$,
was dealt with in Lemma~\ref{ar main}.  To estimate the remaining piece, $|y|\ll M|t|$, one combines
\eqref{combined propig} and Lemma~\ref{ar tail}.
\end{proof}

\section{The double high-to-low frequency cascade}\label{hilo cascade-sec}

In this section, we use the additional regularity provided by Theorem~\ref{glob-sob-thm} to preclude
double high-to-low frequency cascade solutions.

\begin{proposition}[Absence of double cascades]\label{eliminate-ii}  Let $d=2$.  There are no non-zero global
spherically symmetric solutions to \eqref{cubic nls} that are double high-to-low frequency cascades
in the sense of Theorem~\ref{comp}.
\end{proposition}

\begin{proof}
Suppose to the contrary that there is such a solution $u$.  By Theorem~\ref{glob-sob-thm}, $u$ lies in $C_t^0H^1_x(\R\times\R^2)$.
Hence the energy
\begin{equation*}
E(u) = E(u(t)) := \int_{\R^2} \tfrac{1}{2} |\nabla u(t,x)|^2 + \mu\tfrac14|u(t,x)|^{4}\, dx
\end{equation*}
is finite and conserved (see e.g. \cite{caz}).  As we have $M(u) < M(Q)$ in the focusing case,
the sharp Gagliardo-Nirenberg inequality (reproduced here as Theorem~\ref{sGN}) gives
\begin{equation}\label{ume}
\| \nabla u(t) \|_{L^2_x(\R^2)}^2 \sim_u E(u) \sim_u 1
\end{equation}
for all $t \in \R$.  We will now reach a contradiction by proving that $\|\nabla u(t)\|_2\to 0$ along
any sequence where $N(t)\to 0$.  The existence of two such time sequences is guaranteed by the fact that $u$
is a double high-to-low frequency cascade.

Let $\eta > 0$ be arbitrary.  By Definition~\ref{apdef}, we can find $C = C(\eta,u) > 0$ such that
$$ \int_{|\xi| \geq C N(t)} |\hat u(t,\xi)|^2\, d\xi \leq \eta^2$$
for all $t$.  Meanwhile, by Theorem~\ref{glob-sob-thm}, $u\in C_t^0H_x^s(\R\times\R^2)$ for some $s>1$.  Thus,
$$ \int_{|\xi| \geq C N(t)} |\xi|^{2s} |\hat u(t,\xi)|^2\, d\xi \lesssim_{u} 1$$
for all $t$ and some $s>1$.  By H\"older's inequality (or interpolation), we thus obtain
$$ \int_{|\xi| \geq C N(t)} |\xi|^2 |\hat u(t,\xi)|^2\, d\xi \lesssim_{u} \eta^{2(s-1)/s}.$$
On the other hand, from mass conservation and Plancherel's theorem we have
$$ \int_{|\xi| \leq C N(t)} |\xi|^2 |\hat u(t,\xi)|^2\, d\xi \lesssim_{u} C^2 N(t)^2.$$
Summing these last two bounds and using Plancherel's theorem again, we obtain
$$ \| \nabla u(t) \|_{L^2_x(\R^2)} \lesssim_{u} \eta^{(s-1)/s} + C N(t)$$
for all $t$.  As $\eta>0$ is arbitrary and there exists a sequence of times $t_n\to\infty$ such that $N(t_n)\to 0$
($u$ is a double high-to-low frequency cascade), we conclude $\| \nabla u(t_n) \|_2\to 0$.  This contradicts \eqref{ume}.
\end{proof}

\begin{remark}
The role of the dimensional and spherical symmetry hypotheses in Proposition~\ref{eliminate-ii} is to guarantee
regularity, more precisely that $u\in C_t^0H_x^s$ for some $s>1$.  With a few modifications, the argument presented shows
that the hypothesis of spherical symmetry can be replaced by this regularity assumption, as we now explain.
For such regular solutions $u$, we may define the total momentum $\int_{\R^2} \Im( \overline{u} \nabla u )$, which is conserved.
By a Galilean transformation, we can set this momentum equal to zero; thus
$\int_{\R^2} \xi |\hat u(t,\xi)|^2\ d\xi = 0$.  From this, mass
conservation, and the uniform $H^s_x$ bound for some $s>1$, one can show that $\xi(t)
\to 0$ whenever $N(t) \to 0$.  On the other hand, a modification of the above argument gives
$$ 1 \sim_u \|\nabla u(t)\|_2 \lesssim \eta^{(s-1)/s} + C\bigl(N(t) + |\xi(t)|\bigr),$$
which is absurd.

Moreover, if regularity is assumed, the argument presented also shows that even a single-sided cascade is impossible.
By a single-sided cascade we mean a solution with $N(t)$ bounded on a semi-infinite interval, say $[T,\infty)$,
with $\liminf_{t\to\infty}N(t)=0$.  In particular, this argument provides an alternate way of deducing
Corollary~\ref{NO-ss} from Theorem~\ref{ss-sob-thm}.
\end{remark}

\section{Death of a soliton}\label{soliton-sec}
In this section, we use the additional regularity proved in Theorem~\ref{glob-sob-thm} to rule out the third and final enemy,
the soliton-like solution.  Our approach here is similar to that in \cite{merlekenig}.

The key ingredient in disproving the existence of a soliton is a monotonicity formula related to the
virial identity.  Like all formulae of Morawetz type, it expresses the fact that as time passes, the wave
moves from being incoming to being outgoing.  In these crude terms, one expects that if $\vec a(x)$ is a vector field
with $\frac{x}{|x|}\cdot\vec a(x)$ an increasing function of $|x|$, then
\begin{align}\label{general Morawetz}
M_a(t) := 2 \Im \int_{\R^2} \bar u(t,x) \vec a(x)\cdot \nabla u(x,t) \,dx
\end{align}
should be an increasing function of time.  Note that $2 \Im(\bar u \nabla u)$ is the momentum density (or mass current).
We choose $\vec a(x) := x \psi(|x|/R)$, where $\psi$ is a smooth function obeying
$$
\psi(r) = \begin{cases}
1 &: r\leq 1 \\
0 &: r\geq 2
\end{cases}
$$
and $R$ denotes a radius to be chosen momentarily.  For solutions $u$ to \eqref{cubic nls} belonging
to $C^0_t H^1_x$, $M_a(t)$ is a well-defined function.  Indeed,
\begin{align}\label{Morawetz bound}
|M_a(t)| \lesssim R \|u(t)\|_2 \|\nabla u(t)\|_2 \lesssim_u R.
\end{align}
Although $\vec a$ vanishes for large radii, a soliton-like solution (in the sense of Theorem~\ref{comp})
is space localized and so we may still expect monotonicity provided we choose $R\gg 1$.
Substituting our choice of $\vec a$ into \eqref{general Morawetz} leads to the following

\begin{lemma} \label{computation}
\begin{align}
\partial_t M_a(t)&= 8 E(u(t)) \notag \\
    &\quad -\int_{\R^2} \Bigl[ \tfrac{3}{R|x|}\psi'\bigl(\tfrac{|x|}R\bigr) + \tfrac{5}{R^2}\psi''\bigl(\tfrac{|x|}R\bigr)
        + \tfrac{|x|}{R^3}\psi'''\bigl(\tfrac{|x|}R\bigr)\Bigr] |u(t,x)|^2 \,dx \label{M2} \\
    &\quad + 4 \int_{\R^2} \Bigl[ \psi\bigl(\tfrac{|x|}R\bigr) - 1  + \tfrac{|x|}{R}\psi'\bigl(\tfrac{|x|}R\bigr)
        \Bigr] |\nabla u(t,x)|^2 \,dx \label{M3} \\
    &\quad +\mu \int_{\R^2} \Bigl[ 2\psi\bigl(\tfrac{|x|}R\bigr) - 2  + \tfrac{|x|}{R}\psi'\bigl(\tfrac{|x|}R\bigr)
        \Bigr] |u(t,x)|^4 \,dx, \label{M4}
\end{align}
where $E(u)$ is the energy of $u$ as defined in \eqref{energy}.
\end{lemma}

\begin{proposition}[Absence of solitons]\label{eliminate-i}  Let $d=2$.  There are no non-zero global spherically
symmetric solutions to \eqref{cubic nls} that are soliton-like in the sense of Theorem~\ref{comp}.
\end{proposition}

\begin{proof}
Assume to the contrary that there is such a solution $u$.  Then, by Theorem~\ref{glob-sob-thm}, $u\in C_t^0H^s_x$ for any $s>1$.
In particular,
\begin{align}\label{M bound}
|M_a(t)|\lesssim_u R.
\end{align}

Recall that in the focusing case, $M(u)<M(Q)$.  As a consequence, the sharp Gagliardo--Nirenberg inequality
(reproduced here as Theorem~\ref{sGN}) implies that the energy is a positive quantity
in the focusing case as well as in the defocusing case.  Indeed,
\begin{align*}
E(u)\gtrsim_u \int_{\R^2} |\nabla u(t,x)|^2 \,dx >0.
\end{align*}

We will show that \eqref{M2} through \eqref{M4} constitute only a small fraction of $E(u)$.
Combining this fact with Lemma~\ref{computation}, we conclude $\partial_t M_a(t)\gtrsim E(u)>0$,
which contradicts \eqref{M bound}.

We first turn our attention to \eqref{M2}.  This is trivially bounded by
\begin{align}
|\eqref{M2}|\lesssim_u R^{-2}.\label{M2'}
\end{align}

We now study \eqref{M3} and \eqref{M4}.  Let $\eta>0$ be a small number to be chosen later.
By Definition~\ref{apdef} and the fact that $N(t)=1$ for all $t\in \R$, there exists $R=R(\eta,u)$
sufficiently large such that
\begin{align}\label{sol small mass}
\int_{|x|\geq \frac R4} |u(t,x)|^2\, dx\leq \eta
\end{align}
for all $t\in \R$.  Let $\chi$ denote a smooth cutoff to the region $|x|\geq \tfrac R2$; in particular,
$\nabla \chi$ is bounded by $R^{-1}$ and supported where $|x|\sim R$.  As $u\in C_t^0H^s_x$ for some $s>1$,
using interpolation and \eqref{sol small mass}, we estimate
\begin{align}
|\eqref{M3}|
&\lesssim \|\chi\nabla u(t)\|_2^2
\lesssim \|\nabla (\chi u(t))\|_2^2 + \|u(t) \nabla \chi \|_2^2
\lesssim \|\chi u(t)\|_2^{\frac{2(s-1)}s}\|u(t)\|_{H^s_x}^{\frac2s} + \eta  \notag\\
&\lesssim_u \eta^{\frac{s-1}s} + \eta. \label{M3'}
\end{align}

Finally, we are left to consider \eqref{M4}.  Using the same $\chi$ as above together with the Gagliardo--Nirenberg
inequality and \eqref{sol small mass},
\begin{align}
|\eqref{M4}|
&\lesssim\|\chi u(t)\|_4^4
\lesssim \|\chi u(t)\|_2^2 \|\nabla (\chi u(t))\|_2^2
\lesssim_u \eta.\label{M4'}
\end{align}

Combining \eqref{M2'}, \eqref{M3'}, and \eqref{M4'} and choosing $\eta$ sufficiently small depending on $u$
and $R$ sufficiently large depending on $u$ and $\eta$, we obtain
$$
|\eqref{M2}| + |\eqref{M3}| + |\eqref{M4}| \leq \tfrac 1{100} E(u).
$$
This completes the proof of the proposition for the reasons explained in the third paragraph.
\end{proof}

\begin{remark}
The argument just presented can also be used to preclude both the (spherically symmetric) self-similar and
double cascade solutions, once it is known that they are sufficiently regular.  The idea is
to study the behaviour of $M_a(t)$ on intervals $[1,T]$ for $T$ very large; however, in this case we must choose
$R\gg_u T^{1/2}$ rather than $R\gg_u 1$, which was the case above.  A contradiction arises from
$\partial_t M_a(t)>0$ but $|M_a(t)|\lesssim_u T^{1/2}$ on $[1,T]$.

We did not delay the ruling out of these other two types of solutions until now because the arguments
presented earlier are much simpler (and also more robust) than that just sketched.
\end{remark}

\section{A concentration result}\label{cor-sec}
In this section we prove Corollary~\ref{cor:conc}.  The proof is a consequence of Theorem~\ref{main}
and relies on the linear profile decomposition Theorem~\ref{lp}.

It suffices to prove \eqref{conc finite} since \eqref{conc infinite} can be deduced from this using the
pseudoconformal transformation.  To this end, let $u$ be a spherically symmetric solution to \eqref{cubic nls}
that blows up in finite time $0<T^*<\infty$.  Fix a sequence of times $t_n\nearrow T^*$.

After passing to a subsequence if necessary, we have the decomposition
\begin{equation*}
u_n(0):=u(t_n) = \sum_{j=1}^J g^j_n e^{it_n^j\Delta}\phi^j + w^J_n
\end{equation*}
as in Lemma~\ref{lp}.  Let $\psi_n^j$ be the maximal-lifespan solution to \eqref{cubic nls} with initial data
$\psi_n^j(0)=g^j_n e^{it_n^j\Delta}\phi^j$.

From \eqref{un} we obtain the mass decoupling
\begin{equation}\label{massbound}
 \sum_{j=1}^\infty M( \phi^j ) \leq M( u ).
\end{equation}
Thus (reordering the indices $j$ if necessary) we may assume that there exist $\eps>0$ and $J_0\geq 1$ such that
\begin{align}\label{threshold mass}
M(\psi_n^j)=M(\phi^j)\leq M(Q)-\eps \quad \text{for all } j\geq J_0
\end{align}
and $M(\psi_n^j)=M(\phi^j)\geq M(Q)$ for $1\leq j<J_0$.
Combining the small data theory (see Theorem~\ref{local}) with Theorem~\ref{main} and using \eqref{massbound}, we conclude
\begin{align}\label{threshold S}
\sum_{j\geq J_0} \|\psi_n^j\|_{L_{t,x}^4(\R\times\R^2)}^4\lesssim \sum_{j\geq J_0} M(\phi^j)\lesssim_u 1.
\end{align}

Our next job is to find a profile responsible for the finite-time blowup of $u$.  By \eqref{threshold S}, we should
look amongst $\psi_n^j$ with $1\leq j<J_0$.

\begin{lemma}[One bad profile] \label{one bad guy}
There exists $1\leq j_0<J_0$ so that
\begin{equation}\label{one bad E}
M(\phi^{j_0}) \geq M(Q) \quad \text{and} \quad \limsup_{n\to\infty} \|\psi_n^{j_0}\|_{L_{t,x}^4([0, T^*-t_n)\times\R^2)}=\infty.
\end{equation}
Moreover, for each $\eta>0$ there exists $1\leq j_1<J_0$ such that for infinitely many $n$,
the $L_{t,x}^4$ norm of $\psi_n^{j_1}$ reaches $\eta$ first, that is, there exists
$0< \tau_n<T^*-t_n$ such that
\begin{align}\label{moreover}
\bigl\|\psi_n^j\bigr\|_{L_{t,x}^4([0,\tau_n]\times\R^2)}
\leq \bigl\|\psi_n^{j_1}\bigr\|_{L_{t,x}^4([0,\tau_n]\times\R^2)}=\eta
\end{align}
for all $1\leq j<J_0$.
\end{lemma}

\begin{proof}
Suppose that \eqref{one bad E} does not hold for any choice of $j_0$.  Then, by \eqref{threshold S},
\begin{align}\label{finite S}
\sum_{j\geq 1} \|\psi_n^j\|_{L_{t,x}^4([0, T^*-t_n)\times\R^2)}^4 \lesssim_u 1
\end{align}
for $n$ sufficiently large.  We will reach a contradiction by deducing that for $n$ large,
$\|u\|_{L_{t,x}^4([t_n,T^*)\times\R^2)}<\infty$, which is inconsistent with the assumption that $u$ blows up at time $T^*$.

We will obtain our bound on $u$ by combining perturbation theory with bounds on an approximation to $u(t)$, namely,
\begin{equation*}
u^J_n(t) := \sum_{j=1}^J \psi^j_n(t) + e^{it\Delta} w^J_n.
\end{equation*}

We first note that for each $J\geq 1$
\begin{align}\label{same mass}
M( u^J_n(0)- u_n(0) )=0.
\end{align}

Next we establish finite $L_{t,x}^4$ bounds for $u_n^J$ on $[0,T^*-t_n)$.  Using \eqref{crazy}, it is easy to check that
\begin{align}\label{separation}
\lim_{n\to\infty} \bigl\|\psi_n^j \psi_n^{j'}\|_{L_{t,x}^2([0,T^*-t_n)\times\R^2)}=0;
\end{align}
for details see \cite[Lemma~2.7]{keraani-2}.  Combining this with \eqref{finite S}, we estimate
\begin{equation}\label{sun}
\lim_{J \to \infty} \lim_{n \to \infty} \| u^J_n \|_{L_{t,x}^4([0,T^*-t_n)\times\R^2)}^4
=\lim_{J \to \infty} \sum_{j=1}^J \| \psi^j_n \|_{L_{t,x}^4([0,T^*-t_n)\times\R^2)}^4
\lesssim_u 1.
\end{equation}

We now show that $u_n^J$ asymptotically solves \eqref{cubic nls} in the sense that
\begin{align}\label{eq good aprox}
\lim_{J\to\infty} \limsup_{n \to \infty} \| (i \partial_t + \Delta) u^J_n - F(u^J_n) \|_{L^{4/3}_{t,x}([0,T^*-t_n)\times \R^2)} =0.
\end{align}
By the definition of $u^J_n$, we have
$$ (i \partial_t + \Delta) u^J_n = \sum_{j=1}^J F(\psi^j_n)$$
and so, by the triangle inequality, it suffices to show that
$$
\lim_{J\to\infty} \limsup_{n\to\infty} \bigl\| F\bigl(u^J_n-e^{it\Delta}w^J_n\bigr)
        - F\bigl(u^J_n\bigr) \bigr\|_{L^{4/3}_{t,x}([0,T^*-t_n)\times \R^2)}=0
$$
and
$$
    \lim_{n\to\infty} \Bigl\|F\Bigl(\sum_{j=1}^J \psi^j_n\Bigr)
            - \sum_{j=1}^J F(\psi^j_n) \Bigr\|_{L^{4/3}_{t,x}([0,T^*-t_n) \times \R^2)} = 0
$$
for each $J$.  That the first limit is zero follows from H\"older's inequality, \eqref{sln}, and \eqref{sun}.
For the second limit, we use the elementary inequality
$$ \Bigl|F(\sum_{j=1}^J z_j) - \sum_{j=1}^J F(z_j)\Bigr| \lesssim_J \sum_{j \neq j'} |z_j| |z'_j|^2$$
together with H\"older's inequality, \eqref{separation}, and \eqref{sun}.

Next we use the stability result, Lemma~\ref{stab}.  For any small $\delta > 0$, one may choose $J$ sufficiently large
(depending on $\delta$) and then $n$ sufficiently large (depending on $J,\delta$) so that
$$
\|(i \partial_t +\Delta)u^J_n - F(u^J_n)\|_{L^{4/3}_{t,x}([0,T^*-t_n)\times\R^2)}\leq\delta.
$$
Invoking \eqref{same mass} and \eqref{sun} and applying Lemma~\ref{stab} (for $\delta$ chosen small enough), we conclude
that for $n$ sufficiently large,
$$ \|u\|_{L_{t,x}^4([t_n,T^*)\times\R^2)}<\infty.$$
This contradicts the assumption that $u$ blows up at $T^*$.  Thus, there exists at least one $j_0$ satisfying \eqref{one bad E}.
In particular, $J_0\geq 2$.

We now turn our attention to \eqref{moreover}.  By virtue of \eqref{one bad E}, there are infinitely many $n$ such that
$$
\|\psi_n^{j_0}\|_{L_{t,x}^4([0, T^*-t_n)\times\R^2)} >\eta.
$$
For such $n$ there exists $1\leq j(n)<J_0$ that wins the $L_{t,x}^4$ race in the sense of \eqref{moreover}.
Choose $j_1$ to be any value that the sequence $j(n)$ achieves infinitely many times.
\end{proof}

For the remainder of this section, we assume that \eqref{moreover} holds; this amounts to restricting $n$ to a subsequence.
Choosing $\eta>0$ sufficiently small and using \eqref{moreover} and the Strichartz inequality, we conclude
$$
\|e^{it\Delta}\psi_n^{j_1}(0)\|_{L_{t,x}^4([0, \tau_n]\times\R^2)}
\gtrsim \|\psi_n^{j_1}\|_{L_{t,x}^4([0, \tau_n]\times\R^2)}- \|\psi_n^{j_1}\|_{L_{t,x}^4([0, \tau_n]\times\R^2)}^3
\gtrsim \eta.
$$
Combining this with Lemma~\ref{L:B conc} (accounting for scaling), we extract a bubble of concentration that is not
too close to $\tau_n$.  More precisely, there exists $\tau'_n\in (0,\tau_n)$ such that
\begin{align*}
\int_{|x|\lesssim |\tau_n-\tau'_n|^{1/2}} |e^{i\tau'_n \Delta}\psi_n^{j_1}(0)|^2\,dx \gtrsim_{\phi^{j_1},\eta} 1.
\end{align*}
Let $t'_n:=t_n +\tau'_n$.  Using the fact that $\tau_n\in(0, T^*-t_n)$, we derive
\begin{align}\label{bubble}
\int_{|x|\lesssim |T^*-t'_n|^{1/2}} |e^{i\tau'_n \Delta}\psi_n^{j_1}(0)|^2\,dx \gtrsim_{\phi^{j_1},\eta} 1.
\end{align}
As the next lemma shows, we can extract the whole mass of $\psi_n^{j_1}(0)$ by enlarging the diameter of the bubble.

\begin{lemma}[Tightness of profiles]\label{L:comp prof}
Suppose
$$
\int_{|x-x_k|\leq r_k} \bigl|e^{it_k\Delta} \psi \bigr|^2 \,dx \geq \eps
$$
for some $\eps>0$ and sequences $t_k\in\R$, $x_k\in\R^2$, and $r_k>0$.  Then for any sequence $a_k\to\infty$,
\begin{equation}\label{a_k conc}
\int_{|x|\leq a_k r_k} \bigl|e^{it_k\Delta} \psi \bigr|^2 \,dx \to M(\psi).
\end{equation}
\end{lemma}

\begin{proof}
Without loss of generality, we may assume that $t_k\to t_\infty\in[-\infty,\infty]$.  We first consider
the case of a finite limit.  In this case, $e^{it_k\Delta} \psi \to e^{it_\infty\Delta} \psi$ in $L^2$
and so
$$
\liminf_{k\to\infty} \int_{|x-x_k|\leq r_k} \bigl|e^{it_\infty\Delta} \psi \bigr|^2 \,dx \geq \eps
$$
This implies that $\liminf r_k >0$ (also $\limsup |x_k|/r_k <\infty$) and so
$$
\lim_{k\to\infty} \int_{|x|\leq a_k r_k} \bigl|e^{it_\infty\Delta} \psi \bigr|^2 \,dx = M(\psi).
$$
Using $e^{it_k\Delta} \psi \to e^{it_\infty\Delta} \psi$ once again finishes the argument.

In the case that $t_k\to \pm\infty$, we use the Fraunhofer formula: for $t\to\pm\infty$,
\begin{equation}\label{Fraun}
\bigr\| [e^{it\Delta} \psi](x) - (2it)^{-1} e^{i|x|^2/4t} \hat\psi\bigl(\tfrac{x}{2t}\bigr) \bigl\|_{L^2_x}\to 0.
\end{equation}
It allows us to conclude
$$
\liminf_{k\to\infty} \int_{|y-y_k|\leq r_k / |2t_k|} \bigl|\hat\psi(y) \bigr|^2 \,dy \geq \eps,
$$
where $y_k=x_k/(2t_k)$; this shows that $\liminf r_k/|t_k| >0$.  Thus
$$
\lim_{k\to\infty} \int_{|y|\leq a_k r_k / |2t_k|} \bigl|\hat\psi(y) \bigr|^2 \,dy = M(\psi),
$$
which leads to \eqref{a_k conc} via another application of \eqref{Fraun}.
\end{proof}

Returning to \eqref{bubble} and applying Lemma~\ref{L:comp prof} (accounting for scaling again), we obtain
\begin{align}\label{big bubble}
\lim_{n\to\infty} \int_{|x|\leq R_n} |e^{i\tau'_n \Delta}\psi_n^{j_1}(0)|^2\,dx = M(\psi_n^{j_1})= M(\phi^{j_1})\geq M(Q)
\end{align}
for any sequence $R_n$ obeying $(T^*-t_n')^{-1/2} R_n \to \infty$.  It remains to show that a similar bubble
can be found in $u(t_n')$.  This will be effected using perturbation theory (and asymptotic orthogonality)
in much the same way as it was used to prove \eqref{one bad E}.
This time we approximate $u_n(t)$ by $u_n^J$ on $[0,\tau_n]$.
Again $u^J_n(0)=u_n(0)$.  To bound the $L_{t,x}^4$ norm of $u_n^J$ on $[0,\tau_n]$
we argue as for \eqref{sun} using \eqref{threshold S} and \eqref{moreover} in place of \eqref{finite S}.
Repeating the argument that proved \eqref{eq good aprox} we obtain
\begin{align*}
\lim_{J\to\infty} \limsup_{n \to \infty}
    \| (i \partial_t + \Delta) u^J_n - F(u^J_n) \|_{L^{4/3}_{t,x}([0,\tau_n]\times \R^2)} = 0.
\end{align*}
Therefore, by Lemma~\ref{stab}, we conclude
\begin{align}\label{close}
\lim_{J\to\infty}\limsup_{n\to\infty}\|u(t'_n)-u_n^J(\tau'_n)\|_{L_x^2}=0.
\end{align}

To continue, we note that by \eqref{crazy},
\begin{align}\label{got you}
\lim_{J\to\infty}\limsup_{n\to\infty} \bigl| \langle u_n^J(\tau'_n), \psi_n^{j_1}(\tau'_n)\rangle_{L_x^2} \bigr|
    = M(\psi_n^{j_1}) = M(\phi^{j_1})\geq M(Q).
\end{align}
Here, we also used the fact that $w_n^J\rightharpoonup 0$ weakly in $G_\rad\backslash L_x^2$ by virtue of \eqref{sln}.
Putting together \eqref{close} and \eqref{got you}, we obtain
\begin{align*}
\limsup_{n\to\infty} \bigl| \langle u(t'_n), \psi_n^{j_1}(\tau'_n)\rangle_{L_x^2}\bigr| \geq M(\phi^{j_1}).
\end{align*}
By \eqref{moreover} and the Strichartz inequality, this implies
\begin{align*}
\limsup_{n\to \infty} \bigl| \langle u(t'_n), e^{i\tau'_n\Delta} \psi_n^{j_1}(0)\rangle_{L_x^2}\bigr|  \geq M(\phi^{j_1}) -\eta
\end{align*}
provided $\eta>0$ is sufficiently small.  Combining this with \eqref{big bubble}, we conclude
$$
\limsup_{n\to \infty} \int_{|x|\leq R_n} |u(t_n',x)|^2\, dx\geq  M(Q) -\eta.
$$
Sending $\eta$ to zero completes the proof of Corollary~\ref{cor:conc}.

\end{document}